\documentclass[11pt]{amsart}

\usepackage{amsfonts,amssymb,amsthm,eucal,color,bbm,verbatim,graphicx,hyperref}
\usepackage[french,english]{babel}
\usepackage{frcursive}

\usepackage[T1]{fontenc}
\usepackage{aurical}

\usepackage[margin=1.25in]{geometry}
\calclayout

\newtheorem{thm}{Theorem}
\newtheorem{cor}[thm]{Corollary}
\newtheorem{lemma}[thm]{Lemma}
\newtheorem{prop}[thm]{Proposition}

\theoremstyle{definition}

\newcommand{\R}{\mathbb{R}}
\newcommand{\F}{\mathbb{F}}
\newcommand{\E}{\mathbb{E}}

\newcommand{\Prob}{\mathbb{P}}
\newcommand{\N}{\mathbb{N}}

\newcommand{\C}{\mathbb{C}}

\newcommand{\ds}{\displaystyle}
\DeclareMathOperator{\tr}{tr}

\newcommand{\inprod}[2]{\left\langle #1, #2 \right\rangle}
\renewcommand{\Re}{\operatorname{Re}}
\renewcommand{\Im}{\operatorname{Im}}
\newcommand{\abs}[1]{\left\vert #1 \right\vert}
\newcommand{\norm}[1]{\left\Vert #1 \right\Vert}
\newcommand{\bnorm}[1]{\bigl\Vert #1 \bigr\Vert}

\newcommand{\eps}{\varepsilon}

\DeclareMathOperator{\var}{Var}

\newcommand{\Set}[2]{\left\{#1 \mathrel{} \middle| \mathrel{} #2
  \right\}}

\DeclareMathOperator{\diag}{diag}
\newcommand{\Normal}[2]{\mathcal{N}\left(#1,#2 \right)}
\newcommand{\srank}[1]{\,\mbox{\Fontauri{sr}}\,(#1)}
\newcommand{\llambda}{\underline{\lambda}}

\newcommand{\Unitary}[1]{\mathbb{U}\left(#1\right)}

\newcommand{\Orthogonal}[1]{\mathbb{O}\left(#1\right)}
\newcommand{\SOrthogonal}[1]{\mathbb{SO}\left(#1\right)}

\newcommand{\sphere}[1]{\mathbb{S}^{#1}}

\newcommand{\ind}[1]{\mathbbm{1}_{#1}}

\allowdisplaybreaks[3]

\author{Elizabeth S.\ Meckes}

\author{Mark W.\ Meckes}

\address{Department of Mathematics, Applied Mathematics, and
  Statistics, Case Western Reserve University, 10900 Euclid Ave.,
  Cleveland, Ohio 44106, U.S.A.}

\email{elizabeth.meckes@case.edu}

\address{Department of Mathematics, Applied Mathematics, and
  Statistics, Case Western Reserve University, 10900 Euclid Ave.,
  Cleveland, Ohio 44106, U.S.A.}

\email{mark.meckes@case.edu}

\title[Random matrices and random quantum
states]{Random matrices with prescribed eigenvalues and
  expectation values for random quantum states}

\begin{document}

\maketitle

\begin{abstract}
  Given a collection $\llambda=\{\lambda_1, \dots, \lambda_n\} $ of
  real numbers, there is a canonical probability distribution on the
  set of real symmetric or complex Hermitian matrices with eigenvalues
  $\lambda_1,\ldots,\lambda_n$.  In this paper, we study various
  features of random matrices with this distribution.  Our main
  results show that under mild conditions, when $n$ is large, linear
  functionals of the entries of such random matrices have
  approximately Gaussian joint distributions.  The results take the
  form of upper bounds on distances between multivariate
  distributions, which allows us also to consider the case when the
  number of linear functionals grows with $n$.  In the context of
  quantum mechanics, these results can be viewed as describing the
  joint probability distribution of the expectation values of a family
  of observables on a quantum system in a random mixed state.  Other
  applications are given to spectral distributions of submatrices, the
  classical invariant ensembles, and to a probabilistic counterpart of
  the Schur--Horn theorem, relating eigenvalues and diagonal entries
  of Hermitian matrices.
\end{abstract}

\tableofcontents

\section{Introduction}

Let $\lambda_1 \le \lambda_2 \le \dots \le \lambda_n$ be real numbers,
and let $M_n^{\R}(\llambda)$ denote the family of real symmetric
$n\times n$ matrices with eigenvalues
$\llambda=\{\lambda_1, \dots, \lambda_n\}$ (with multiplicity).  The
orthogonal group $\Orthogonal{n}$ acts transitively on $M_n^{\R}(\llambda)$
by conjugation, and from this action $M_n^{\R}(\llambda)$ inherits a
canonical probability measure.  A random matrix chosen according to
this probability measure is distributed as $U \Lambda U^t$, where
$\Lambda = \diag(\lambda_1, \dots, \lambda_n)$, and $U$ is chosen
according to the Haar probability measure on
$\Orthogonal{n}$. Likewise, the family $M_n^{\C}(\llambda)$ of complex
Hermitian matrices with eigenvalues $\lambda_1,\dots, \lambda_n$
possesses a canonical probability measure which is the distribution of
the random matrix $U \Lambda U^*$, where $U$ is now chosen according
to the Haar probability measure on the unitary group $\Unitary{n}$.

In this paper we consider the asymptotic behavior of $d$-dimensional
marginals of these probability measures (sometimes referred to as
isospectral distributions) when $n$ is large.  Such marginals include
in particular the joint distributions of collections entries of random
matrices of the form $U\Lambda U^t$ or $U\Lambda U^*$. Our results
give upper bounds on distances between these marginal distributions
and multivariate Gaussian distributions. Such quantitative estimates
hold for fixed, finite (but large) $n$, which in turn allows us to
consider, as $n \to \infty$, how quickly $d$ may grow with $n$ such
that the Gaussian behavior of $d$-dimensional marginals is
preserved. In typical situations we will see that \emph{all} such marginals
are asymptotically Gaussian, as long as $d \ll \sqrt{n}$;
for many choices of $\Lambda$ and many marginals,
it is in fact only necessary that $d \ll n^{3/2}$.

Our results reverse the situation from classical random matrix theory, which
begins by specifying the joint distributions of the entries of a
random matrix and investigates the resulting joint distribution of the
eigenvalues. This inverse approach reveals a new form of universality:
marginals of high-dimensional random matrices with nearly any
arrangement of prescribed eigenvalues are indistinguishable from
marginals of the Gaussian orthogonal or unitary ensemble. This in turn
puts limits on how far one can hope to extend classical
universality to random matrix ensembles with dependent entries. Weakly
correlated entries, even weakly correlated nearly-Gaussian entries,
turn out to be consistent with almost any kind of spectral behavior.

If the eigenvalues $\lambda_1, \dots, \lambda_n$ are nonnegative and
$\sum_{i=1}^n \lambda_i = 1$, then our results have an important
interpretation in terms of quantum mechanics.  In this case
$\rho = U \Lambda U^*$ is a random density matrix, representing a
mixed state with weights $\{\lambda_i\}$ of a quantum system with an
$n$-dimensional state space $\mathcal{H} = \C^n$. (See for example
\cite{Wootters,BeZy,Nechita} for general information on various models
of random mixed quantum states, and \cite{OsKu,OsAuGoKoAcLe} for
earlier work considering this precise model.)  If $B_1, \dots, B_d$
are $d$ linearly independent observables on $\mathcal{H}$, then the
joint probability distribution of their expectation values in the
state $\rho$ is a $d$-dimensional marginal of the distribution of
$\rho$, and therefore, by our results, is distributed approximately as
a $d$-dimensional Gaussian random vector.  Moreover, if
$\mathcal{H} = \mathcal{H}_1 \otimes \dots \otimes \mathcal{H}_d$ and
each $B_i$ arises as an observable on $\mathcal{H}_i$ (that is, we are
considering a compound system and separately observing the component
systems), then the expectation values $\langle B_i \rangle$ are
approximately distributed as \emph{uncorrelated} jointly Gaussian
random variables. Besides random density matrices with fixed
eigenvalues, our results cover induced random density matrices, which
arise as quantum marginals of random pure states on compound quantum
systems.

Random matrices of the form $U \Lambda U^*$ are familiar in free
probability, where stochastically independent matrices of this type
are used to asymptotically model freely independent noncommutative
random variables.  Free probability is chiefly concerned with the
asymptotic spectral distributions of functions of families of such
random matrices (where $\Lambda = \Lambda_n$ has a known limiting
spectral measure when $n \to \infty$), in contrast to our interest
here in linear projections. These two lines of inquiry intersect in
the study of the limiting spectral behavior of submatrices, sometimes
referred to as free compressions (see e.g.\ \cite{NiSp}).  By
combining our main result with a quantitative version of Wigner's
semicircle law, we are able to deduce that submatrices of a range of
sizes have asymptotically semicircular spectral distributions.

In addition to deterministic $\Lambda$, we can by conditioning allow
$\Lambda$ to be random and independent of $U$.  Such a construction
produces exactly that class of distributions on real symmetric
(respectively, complex Hermitian) matrices which are invariant under
orthogonal (respectively, unitary) conjugation, including the so-called
unitarily invariant ensembles.

Finally, specializing our results to the diagonal entries of
$U \Lambda U^*$ lets us investigate in a natural way the ``typical''
relationship between the eigenvalues and diagonal entries of a real
symmetric or complex Hermitian matrix.  In this way we find, in
Theorem \ref{T:Schur-Horn}, a probabilistic counterpart of the
Schur--Horn theorem, which characterizes pairs of $n$-tuples which
occur as the eigenvalues and diagonal entries of some real symmetric
matrix. Theorem \ref{T:Schur-Horn} is thus a Hermitian analogue of the single ring theorem
\cite{GuKrZe,RuVe}, which can be viewed as a probabilistic counterpart
of the Weyl--Horn theorem.

Our main technical tool is a multivariate version of Stein's method
for normal approximation developed in \cite{ChMe}, which was used in
that paper to prove similar results for marginals of the entries of
Haar-distributed random matrices; see \cite{DoSt,Stolz,JoSmWe,JoSm}
for additional applications of this method in random matrix theory.
There are at least two other well-established approaches to studying
random matrices of the form considered here.  First, the so-called
Harish-Chandra--Itzykson--Zuber integral formula \cite{Harish-Chandra}
gives an exact expression for the Fourier transform of the
distribution of $U \Lambda U^*$.  Second, the Weingarten calculus
developed in \cite{Collins,CoSn,CoNe-rqc} gives combinatorial
expressions for joint moments or joint cumulants of entries
Haar-distributed random matrices, thus more indirectly of matrices of
our form. The chief advantage of Stein's method over these other
approaches are that it automatically yields quantitative bounds on
distances between distributions. We recall that the explicit
dependence of our bounds on the projection dimension $d$ are a crucial
aspect of our results.  Such quantitative bounds are rather delicate
work to deduce using Fourier-analytic methods (and the HCIZ formula is
itself awkwardly suited for high-dimensional analysis; cf.\
\cite{GuZe-lda,GuMa}), and are generally out of reach using moments.
Nevertheless, it may be that our main results could be sharpened
somewhat by combining our Stein's method approach with the Weingarten
calculus; see the comments after the proofs of Theorems \ref{T:DL-T0-C}
and \ref{T:R1-T0-C} in Section \ref{S:proofs-main}.

In the remainder of this introduction we will state our results and
expand on the above discussion, deferring the proofs to the later
sections.

\subsection{Main results}

We first establish some notation and terminology.

We denote by $M_n(\F)$ the space of $n \times n$ matrices with entries
in the field $\F$. We denote by $M_n^{sa}(\R)$ the space of real
symmetric matrices and by $M_n^{sa}(\C)$ the space of complex
Hermitian matrices.  For either $\F=\R$ or $\C$, recall that
$M_n^{sa}(\F)$ is a real vector space which is equipped with the
(real) Hilbert--Schmidt inner product
$\inprod{A}{B} = \tr(AB^*) = \tr (AB)$.

For $A \in M_n(\C)$ and $1 \le p \le \infty$, $\norm{A}_p$ denotes the
Schatten $p$-norm of $A$, which is the $\ell_p$ norm of the sequence
of singular values of $A$.  If $A \in M_n^{sa}(\C)$ then $\norm{A}_p$
is equal to the $\ell_p$ norm of the sequence of eigenvalues of
$A$.  In the special cases $p = 2, \infty$ this reduces to the
Hilbert--Schmidt and operator norms, which we also denote by
\[
\norm{A}_{HS} = \norm{A}_2 = \sqrt{\tr (AA^*)}
\]
and
\[
\norm{A}_{op} = \norm{A}_\infty
\]
respectively.  The stable rank (also called numerical rank) of a
matrix $A$ is $\srank{A} = \frac{\norm{A}_{HS}^2}{\norm{A}_{op}^2}$.
Note that for any nonzero matrix $A$,
$1 \le \srank{A} \le \operatorname{rank} A$; roughly speaking,
$\srank{A}$ is small when $A$ is a small perturbation of a matrix with
small rank.

Results below
are often most easily formulated for traceless matrices; given $B \in
M_n^{sa}(\C)$, we will use the notation $\widetilde{B}$ for the
traceless recentering of $B$, i.e.,
\begin{equation}\label{E:traceless}
\widetilde{B} = B - \frac{1}{n} (\tr B) I_n.
\end{equation}

If $X$ is a random vector in a real Hilbert space $\mathcal{H}$ with
probability distribution $\mu$ and $V$ is a $d$-dimensional subspace
of $\mathcal{H}$, then the marginal of $\mu$ on $V$ is the
distribution of $\pi_V(X)$, where $\pi_V$ denotes orthogonal
projection. Such a marginal can be represented in coordinates by the
vector $(\inprod{X}{v_1},\dots, \inprod{X}{v_d}) \in \R^d$, where
$v_1, \dots, v_d$ is a fixed orthonormal basis of $V$. 
 
We denote by $\Normal{\mu}{\Sigma}$ the Gaussian distribution on
$\R^d$ with mean $\mu \in \R^d$ and covariance matrix
$\Sigma \in M_d^{sa}(\R)$.  A standard Gaussian random vector, i.e., one with $\mu=0$ and $\Sigma=I_d$, is denoted by $g=(g_1,\ldots,g_d)$ throughout.  The Gaussian Unitary Ensemble (GUE) can be
defined as a standard Gaussian random vector in the Hilbert space
$M_n^{sa}(\C)$, and the Gaussian Orthogonal Ensemble (GOE) can be
defined as a $\sqrt{2}$ times a standard Gaussian random vector in
$M_n^{sa}(\R)$.

We will mostly quantify the approximation of probability measures using the
Wasserstein (or Kantorovich) metrics (see e.g.\ \cite{Villani} for
extensive discussion of these metrics). If $X$ and $Y$ are random
vectors in $\R^d$, the $L_p$-Wasserstein distance between them (or
more properly, between their distributions) may be defined as
\[
W_p(X,Y) = \inf\Set{\left(E\norm{Z_1-Z_2}^p\right)^{1/p}}{Z_1
  \overset{d}{=} X, Z_2 \overset{d}{=} Y};
\]
that is, the infimum is taken over couplings of $X$ and $Y$.  By the
Kantorovich--Rubenstein Theorem, the $L_1$-Wasserstein distance is the
same as
\[
W_1 (X,Y) = \sup_{\abs{f}_L \le 1} \abs{\E f(X) - \E f(Y)},
\]
where $\abs{f}_L$ denotes the Lipschitz constant of a function
$f:\R^d \to \R$.  We will occasionally also use the total variation
distance
\[
d_{TV}(X,Y) = \frac{1}{2} \sup_{\norm{f}_{\infty} \le 1} \abs{\E f(X)
  - \E f(Y)}.
\]

\medskip

Theorem \ref{T:DL-T0-R} below is our main result in the real symmetric
case. Observe that if we fix the eigenvalues of a matrix, then we have
fixed its trace.  It is therefore natural (and convenient for the
proof) to begin by considering only coefficient matrices that lie in
the subspace $\Set{B \in M_n^{sa}(\R)}{\tr B = 0}$; we will remove this
restriction below.

\begin{thm}
  \label{T:DL-T0-R}
  Let $\Lambda \in M_n(\R)$ be a nonscalar diagonal matrix, and define
  $A = U \Lambda U^t$, where $U$ is a Haar-distributed random matrix
  from $\Orthogonal{n}$.  Let $B_1, \dots, B_d \in M_n^{sa}(\R)$
  satisfy $\tr B_i B_j = \delta_{ij}$ and $\tr B_i = 0$, and define
  the random vector $X \in \R^d$ by $X_i = \tr A B_i$.  

  Then
  \begin{equation}
    \label{E:DL-T0-R-W}
    \begin{split}
    W_1\left(\frac{\sqrt{(n-1)(n+2)}}{\sqrt{2} \bnorm{\widetilde{\Lambda}}_{HS}}
      X, g\right) 
    & \le 8 \sqrt{2} \frac{\sqrt{n-1}(n+2)}{n}
    \frac{\bnorm{\widetilde{\Lambda}}_{op}^2}{\bnorm{\widetilde{\Lambda}}_{HS}^2}
    \sum_{i=1}^d \norm{B_i}_{op}^2
    \\
    & = 8 \sqrt{2}
    \frac{\sqrt{n-1}(n+2)}{n \srank{\widetilde{\Lambda}}}
    \sum_{i=1}^d \frac{1}{\srank{B_i}},
    \end{split}
  \end{equation}
where $\widetilde{\Lambda}=\Lambda-\frac{1}{n}(\tr(\Lambda))I_n$.

  If $d=1$ then also 
  \[
    d_{TV} \left(\frac{\sqrt{(n-1)(n+2)}}{\sqrt{2}
        \bnorm{\widetilde{\Lambda}}_{HS}} X, g\right) \le 16 \sqrt{2}
    \frac{\sqrt{n-1}(n+2)}{n}
    \frac{\bnorm{\widetilde{\Lambda}}_{op}^2}{\bnorm{\widetilde{\Lambda}}_{HS}^2}
    \norm{B_1}_{op}^2.
  \]
\end{thm}

We emphasize that in Theorem \ref{T:DL-T0-R} and many of the results below, the bounds given are completely explicit and hold as soon as $n\ge 2$.

Theorem \ref{T:DL-T0-R} shows that marginals of the distribution of
the random matrix $A$ are close to  those of the GOE. Indeed, if $V$ is
a $d$-dimensional subspace of $\Set{B \in M_n^{sa}(\R)}{\tr B = 0}$
and $G$ is an $n \times n$ GOE random matrix, then
$\frac{1}{\sqrt{2}}\pi_V(G)$ has a standard Gaussian distribution on
$V$. Thus the left-hand side of the inequality \eqref{E:DL-T0-R-W} is
precisely
\[
W_1 \left(\frac{\sqrt{(n-1)(n+2)}}{\sqrt{2}
    \bnorm{\widetilde{\Lambda}}_{HS}} \pi_V(A),
  \frac{1}{\sqrt{2}}\pi_V(G)\right) = \frac{1}{\sqrt{2}} W_1
\left(\frac{\sqrt{(n-1)(n+2)}}{\bnorm{\widetilde{\Lambda}}_{HS}}
  \pi_V(A), \pi_V(G)\right).
\]

The following application of Theorem \ref{T:DL-T0-R} is illustrative.
Suppose that $n$ is even, and that $\frac{n}{2}$ of the diagonal
entries of $\Lambda$ are equal to $\sqrt{n}$, and $\frac{n}{2}$ of
them are equal to $-\sqrt{n}$.  Theorem \ref{T:DL-T0-R} implies that
\begin{equation}
  \label{E:nice-bound}
W_1 \left(\pi_V(A), \pi_V(G)\right) \le
C \frac{d}{\sqrt{n}}
\end{equation}
for some absolute constant $C$, and any $d$-dimensional subspace $V$.
Thus all the $d$-dimensional marginals of $A$ (on the subspace of
trace-zero matrices) are close to Gaussian as long as
$d \ll \sqrt{n}$, although the spectrum of $A$ is very different from
the spectrum of the GOE.  (As mentioned above, the trace-zero
restriction will be removed below.)

Alternatively, suppose that the $B_i$ all have stable
rank of the order of $n$. In that case the same bound as in
\eqref{E:nice-bound} applies for an arbitrary nonzero $\Lambda$; for
the specific $\Lambda$ suggested above, the bound  improves to
\[
W_1 \left(\pi_V(A), \pi_V(G)\right) \le
C \frac{d}{n^{3/2}}.
\]
This improvement in the normal approximation of the random vector $X$
for essentially high-rank coefficient matrices $B_i$ is a quadratic
counterpart to a phenomenon uncovered in \cite{KeMeSi}.  It is shown
there that the best rate of convergence for $\Re \tr BU$, where $U$ is
a Haar-distributed random matrix from $\Unitary{n}$, depends on the
asymptotic behavior of the singular values of $B$.  As in our case,
the slowest convergence rate occurs for $\operatorname{rank} B = 1$.

Theorem \ref{T:DL-T0-C} is our main result in the complex Hermitian
case. 

\begin{thm}
  \label{T:DL-T0-C}
  Let $\Lambda \in M_n(\R)$ be a nonscalar diagonal matrix, and define
  $A = U \Lambda U^*$, where $U$ is a Haar-distributed random matrix
  from $\Unitary{n}$.  Let $B_1, \dots, B_d \in M_n^{sa}(\C)$ satisfy
  $\tr B_i B_j = \delta_{ij}$ and $\tr B_i = 0$, and define the random
  vector $X \in \R^d$ by $X_i = \tr A B_i$.  

  Then
  \begin{equation}
    \label{E:DL-T0-C-W}
    W_1 \left(\frac{\sqrt{n^2-1}}{\bnorm{\widetilde{\Lambda}}_{HS}} X,g\right) 
    \le 8 \sqrt{n}
    \frac{\bnorm{\widetilde{\Lambda}}_{op}^2}{\bnorm{\widetilde{\Lambda}}_{HS}^2}
    \sum_{i=1}^d \norm{B_i}_{op}^2
    = \frac{8 \sqrt{n}}{\srank{\widetilde{\Lambda}}}
    \sum_{i=1}^d \frac{1}{\srank{B_i}},
  \end{equation}
  where $\widetilde{\Lambda}=\Lambda-\frac{1}{n}(\tr(\Lambda))I_n$.

  If $d=1$ then also
  \begin{equation*}
    d_{TV} \left(\frac{\sqrt{n^2-1}}{\bnorm{\widetilde{\Lambda}}_{HS}} X,g\right) 
    \le 16\sqrt{n}
    \frac{\bnorm{\widetilde{\Lambda}}_{op}^2}{\bnorm{\widetilde{\Lambda}}_{HS}^2}
    \norm{B_1}_{op}^2.
  \end{equation*}
\end{thm}

As above, Theorem \ref{T:DL-T0-C} shows that marginals of the
distribution of the complex Hermitian version of the random matrix $A$
are close to those of the GUE: if $V$ is a $d$-dimensional
subspace of $\Set{B \in M_n^{sa}(\C)}{\tr B = 0}$ and $G$ is now an
$n \times n$ GUE random matrix, then the left-hand side of
\eqref{E:DL-T0-C-W} is equal to
\[
W_1 \left(\frac{\sqrt{n^2-1}}{\norm{\Lambda}_{HS}} \pi_V(A),
  \pi_V(G)\right).
\]
The same specific example discussed above (eigenvalues of $A$ evenly
split between $\pm \sqrt{n}$) serves as a useful prototype for Theorem
\ref{T:DL-T0-C} as well.

\medskip
 
For the sake of brevity, from this point on we will explicitly state
our results only in the complex Hermitian version, although all the
results below have real symmetric counterparts which differ only in
the constants which appear.  We will also omit further estimates in
total variation for the univariate case.  

\medskip

As mentioned above, the assumptions on the coefficient matrices $B_i$
can be removed by a suitable affine transformation.

\begin{cor}
  \label{T:DL-B-C}
  Let $\Lambda \in M_n(\R)$ be a nonscalar diagonal matrix, and define
  $A = U \Lambda U^*$, where $U$ is a Haar-distributed random matrix
  from $\Unitary{n}$.  Let $B_1, \dots, B_d \in M_n^{sa}(\C)$ and let
  $\widetilde{B_j} = B_j - \frac{1}{n} (\tr B_j) I_n$.  Define
  $\Sigma \in M_d(\R)$ by
  $\Sigma_{ij} = \tr \widetilde{B_i} \widetilde{B_j}$ and $v \in \R^d$
  by $v_i = \frac{1}{n} (\tr \Lambda) (\tr B_i)$.  Define the random
  vector $X \in \R^d$ by $X_i = \tr A B_i$.  Then
  \[
    W_1\left(X,\frac{\bnorm{\widetilde{\Lambda}}_{HS}}{\sqrt{n^2-1}}
      \Sigma^{1/2} g + v\right) \le
    \frac{8 d}{\sqrt{n-1}} \bnorm{\Sigma^{1/2}}_{op} \frac{\bnorm{\widetilde{\Lambda}}_{op}^2
    }{\bnorm{\widetilde{\Lambda}}_{HS}}.
  \]
\end{cor}

In order to keep the statement of Corollary
\ref{T:DL-B-C} simple, we have made use of the trivial estimate
$\srank{B_i} \ge 1$ (thus replacing the sum in the last expression in
\eqref{E:DL-T0-C-W} with $d$). Corollary \ref{T:DL-B-C} is only applied
in the present paper in the proof of Theorem \ref{T:submatrix} below,
where $\srank{B_i} \le 2$ for each $i$, so there is no essential loss
in making this simplification.  For applications involving $B_i$ with
large stable rank, modifying the proof of Corollary \ref{T:DL-B-C} to
make use of any available special structure of the problem at hand
seems likely to be more useful than attempting to formulate a general-purpose off-the-shelf result.

In the case that $\Lambda$ has rank $1$ (in
which case the results above only have interesting content when all of
the coefficient matrices $B_i$ have large stable rank), the bound on
normal approximation from Theorem \ref{T:DL-T0-C} can be sharpened
somewhat. By rotation invariance, we may assume in
this case that $\Lambda = e_1 e_1^*$, where $e_1$ is the first
standard basis vector of $\C^n$. Then $Z = Ue_1$ is uniformly
distributed on the unit sphere of $\C^n$, and
\[
  \tr A B_j = \tr Z Z^* B_j = \inprod{B_j Z}{Z}.
\]
The rank 1 case thus amounts to studying the joint
distribution of quadratic forms on the unit sphere.

\begin{thm}
  \label{T:R1-T0-C}
  Let $Z=(Z_1,\ldots,Z_n)$ be uniformly distributed on the complex
  unit sphere.  Let $B_1, \dots, B_d \in M_n^{sa}(\C)$ satisfy
  $\tr B_j B_k = \delta_{jk}$ and $\tr B_j = 0$  For
  $j=1,\ldots,d$, let $X_j:=\inprod{B_jZ}{Z}$.  There is a universal
  constant $C$ such that
  \begin{equation}
    \label{E:R1-T0-C}
  W_1\left(\sqrt{n(n+1)}X,g\right)\le
  C \sum_{j=1}^d\norm{B_j}_4^{2}
  \le C \sum_{j=1}^d\frac{1}{\sqrt{\srank{B_j}}}.
  \end{equation}

\end{thm}

As in the proof of Corollary \ref{T:DL-B-C}, if desired, one could remove the
assumptions on the $B_j$ using standard linear algebraic techniques.

The second, weaker upper bound in Corollary
\ref{T:R1-T0-C} shows that in this setting, the asymptotic behavior of projections is
Gaussian if each coefficient matrix $B_j$ has stable rank of larger
order than $d^2$; in particular, if $d$ is fixed and the stable ranks
of the $B_j$ grow without bound.  Theorem \ref{T:DL-T0-C} applied to
this setting would require that the stable ranks be of larger order
than $d \sqrt{n}$.  Moreover, the first, stronger upper bound in
Corollary \ref{T:R1-T0-C} shows that the same conclusion holds if
some of the $B_j$ have a small number of relatively large eigenvalues.

In principle, it should be possibly to modify our proofs to
prove a result which encompasses both Theorems \ref{T:DL-T0-C} and
\ref{T:R1-T0-C} as special cases, although there are technical challenges; see the discussion following the
proof of Theorem \ref{T:R1-T0-C}.

The proofs of Theorem \ref{T:DL-T0-C} (and indications of how to
modify the proof for the real symmetric case in Theorem
\ref{T:DL-T0-R}), Corollary \ref{T:DL-B-C}, and Theorem
\ref{T:R1-T0-C} are given in Section \ref{S:proofs-main} below.

\subsection{Expectation values of observables for random quantum
  states}

As mentioned earlier, the results above have a natural interpretation
in terms of random mixed states of quantum mechanical systems. We will
briefly summarize some basic terminology for readers unfamiliar with
quantum mechanics; see \cite{AuSz,BeZy} for more details.  For
consistency we will continue to use the same linear-algebraic notation
as above, rather than switching to the bra-ket notation typically used
in the context of quantum mechanics.

A \emph{density matrix} is a matrix $\rho \in M_n^{sa}(\C)$ with
nonnegative eigenvalues such that $\tr \rho = 1$.  Equivalently,
$\rho \in M_n^{sa}(\C)$ is a density matrix if $\rho = \tr_2 (\psi \psi^*)$,
where $\psi$ is a unit vector in $\C^n \otimes \C^s \cong \C^{ns}$ for
some $s$, and $\tr_2 : M_{n}(\C) \otimes M_s(\C) \to M_n(\C)$ is the
partial trace defined by $\tr_2(A \otimes B) = (\tr B) A$. A density
matrix $\rho \in M_n(\C)$ represents a mixed state of a quantum
mechanical system modeled on the finite-dimensional Hilbert space
$\mathcal{H} = \C^n$.  

A pure state corresponds to the special case of $\rho = \psi \psi^*$
for a unit vector $\psi \in \C^n$; the vector $\psi$ itself is often
said to represent such a pure state.  A mixed state is thus the
partial trace over $\C^s$ of a pure state in some larger Hilbert space
$\C^n \otimes \C^s$. In this case the factor spaces $\C^n$ and $\C^s$
represent interacting subsystems, and $\rho = \tr_2(\psi \psi^*)$
represents the state of the individual system modeled by $\C^n$, when
the composite system is in the pure state $\psi$; $\rho$ is sometimes
referred to as a \emph{quantum marginal} of $\psi$.

An observable of a quantum mechanical system modeled on $\C^n$ is
represented by a Hermitian matrix $B \in M_n^{sa}(\C)$.  If the system
is in a mixed state represented by the density matrix $\rho$, then the
expectation value of the observable $B$ is $\langle B \rangle = \tr
(\rho B)$; in a pure state $\psi$ this becomes $\langle B \rangle =
\inprod{B \psi}{\psi}$.

We can thus interpret Theorems \ref{T:DL-T0-C} and \ref{T:R1-T0-C} and
Corollary \ref{T:DL-B-C} as statements about the
joint probability distributions of expectation values of observables
of quantum systems in random states. Suppose
$B_1, \dots, B_d \in M_n^{sa}(\C)$ are observables on a quantum system
modeled by the Hilbert space $\C^n$. We assume that the system is in a
mixed state $\rho$ with known eigenvalues $\llambda$, but which is
otherwise unknown; this is reasonably modeled by a random density
matrix $\rho = U\Lambda U^*$ with $U \in \Unitary{n}$
Haar-distributed.  Theorem \ref{T:DL-T0-C} and Corollary
\ref{T:DL-B-C} show that, under certain hypotheses, the random vector
\[
(\langle B_1 \rangle, \dots, \langle B_d \rangle) \in \R^d
\]
has a jointly Gaussian probability distribution. Theorem
\ref{T:R1-T0-C} does the same for a random
pure state $Z$ uniformly distributed in the unit sphere of $\C^n$. 
(Note that the
randomness here comes entirely from the uncertainty in the state
$\rho$; there is no quantum mechanical randomness since we are
considering expectation values of the observables.)

While it is natural to consider the case in which the eigenvalues, but
nothing else, are known, even this level of certainty may not hold in
practice. There are several well-studied probability measures on the
space $n\times n$ of density matrices, among the most important of
which are the so-called \emph{induced measures} $\mu_{n,s}$ for
integer $s \ge 1$ (see \cite{AuSz,BeZy,Nechita,ZySo}). If $Z$ is uniformly
distributed on the unit sphere of $\C^n\otimes\C^s$, $\mu_{n,s}$ is
the distribution of the random density matrix
\[
\rho_{n,s} = \tr_2(ZZ^*) \in M_n(\C).
\]
That is, $\rho_{n,s}$ is a quantum marginal on $\C^n$ of a uniform
random pure state on the composite system modeled by
$\C^n \otimes \C^s$.  In the special case that $s=n$, $\mu_{n,n}$
coincides with normalized Lebesgue measure (usually referred to as
Hilbert--Schmidt measure in this context) on the space of density
matrices.

The following result is an easy application of Theorem
\ref{T:R1-T0-C}. (Since $\mu_{n,s}$ is invariant under unitary
conjugation, one could also approach this via Theorem
\ref{T:RL-B-C} below; however, the approach via Theorem
\ref{T:R1-T0-C} gives a stronger result.)

\begin{thm}\label{T:induced-states} 
  Let $\rho_{n,s}=\tr_2(ZZ^*) $ be a quantum marginal on $\C^n$ of the
  uniform random pure state $Z$ on $\C^n \otimes \C^s$.  Let
  $B_1,\ldots,B_d$ be traceless $n\times n$ Hermitian matrices with
  $\tr(B_jB_k)=\delta_{jk}$.  For $j=1,\ldots,d$, let
  $X_j:=\tr(\rho_{n,s} B_j)$.  There is a universal constant $C$ such
  that
  \[
  W_1\left(\sqrt{n(ns+1)}X,g \right) \le
  \frac{C}{\sqrt{s}} \sum_{j=1}^d\norm{B_j}_4^{2}.
  \]
\end{thm}

Theorem \ref{T:induced-states} is proved in Section
\ref{S:proofs-quantum}.

A particularly important special case is when $\C^n$ is itself a
tensor product $\C^n = \C^{n_1} \otimes \cdots \otimes \C^{n_k}$ and
the $B_j$ have the form
\[
B_j = I_{n_1} \otimes \cdots I_{n_{j-1}} \otimes C_j \otimes
I_{n_{j+1}} \otimes \cdots \otimes I_{n_k}
\]
for some $C_j \in M_{n_j}^{sa}(\C)$. In that case $\C^n$ itself models
a composite system, and each $B_j$ corresponds to an observable acting
on a distinct component system; note that when the $B_j$ are traceless they are
automatically orthogonal with respect to the Hilbert--Schmidt inner
product.

In the case that $V = \Set{B \in M_n^{sa}(\C)}{\tr B = 0}$, $\pi_V(A)
= A - \frac{1}{n} (\tr A) I_n = \widetilde{A}$.  Theorem \ref{T:induced-states}
implies that 
\[
W_1\left(\sqrt{n(ns+1)} \widetilde{\rho_{n,s}}, \widetilde{G}\right) \le C
\frac{n^2}{\sqrt{s}},
\]
where $G$ is an $n\times n$ GUE random matrix.  This recovers the
fact, apparently first observed in \cite[Theorem 6.35(i)]{AuSz}, that
for fixed $n$ and $s\to \infty$, $\widetilde{\rho_{n,s}}$ converges,
after appropriate rescaling, to the traceless GUE.  Moreover, it adds
to this observation a rate of convergence, which allows the earlier result to be
meaningfully extended to the regime $s \gg n^4$.

\subsection{Joint distributions of entries}
\label{S:joint}

We now focus our attention on the matrix entries of $A = U \Lambda
U^*$.  Our results below are stated only for traceless $\Lambda$; extending to the general case is trivial but more complicated to state.

\begin{thm}
  \label{T:DL-entries-C}
  Let $\Lambda \in M_n(\R)$ be a nonzero diagonal matrix with
  $\tr \Lambda = 0$, and define $A = U\Lambda U^*$, where $U$ is a
  Haar-distributed random matrix from $\Unitary{n}$.  Let $X \in \R^d$
  be a random vector whose entries are distinct choices among the diagonal entries of $A$, the
  real parts of the above-diagonal entries of $A$ scaled up by $\sqrt{2}$, and
  the imaginary parts of the above-diagonal
  entries of $A$ scaled up by $\sqrt{2}$.  Then
  \[
  W_1\left(\frac{\sqrt{n^2-1}}{\norm{\Lambda}_{HS}} X, g \right) 
  \le 9 d \sqrt{n} \frac{\norm{\Lambda}_{op}^2}{\norm{\Lambda}_{HS}^2}=\frac{9d\sqrt{n}}{\srank{\Lambda}}.
  \]
\end{thm}

This result in particular gives a direct comparison between principle
submatrices of $A$ and the GUE, as follows.

\begin{thm}
  \label{T:submatrix}
  Let $\Lambda \in M_n(\R)$ be a nonzero diagonal matrix with
  $\tr \Lambda = 0$, and define $A = U\Lambda U^*$, where $U$ is a
  Haar-distributed random matrix from $\Unitary{n}$. Let $B$ be the
  upper-left $k \times k$ truncation of $A$, and let $G$ be a
  $k\times k$ GUE matrix. Then
  \[
  W_1 \left(\frac{\sqrt{n^2-1}}{\norm{\Lambda}_{HS}} B, G \right) \
  \le 18 k^2 \sqrt{n}
  \frac{\norm{\Lambda}_{op}^2}{\norm{\Lambda}_{HS}^2}=\frac{18k^2\sqrt{n}}{\srank{\Lambda}}.
  \]
\end{thm}

Via quantitative versions of the semicircle law for the GUE and
concentration of measure arguments, this allows us to approximate the
spectral measure of a suitably scaled version of $B$ by the
semicircle law.

\begin{thm}
  \label{T:DL-semicircle}
  Let $\Lambda \in M_n(\R)$ be a nonzero diagonal matrix with
  $\tr \Lambda = 0$, and define $A = U\Lambda U^*$, where $U$ is a
  Haar-distributed random matrix from $\Unitary{n}$. Let $B$ be the
  upper-left $k \times k$ truncation of $A$, and let
  $M:=\frac{\sqrt{n^2-1}}{\norm{\Lambda}_{HS}} B$.  Let $\rho_{sc}$
  denote the semicircular distribution, with density $\sqrt{4-t^2}$ on $[-2,2]$.  Then 
\begin{align*}\E W_1 (\mu_{k^{-1/2}M}, \rho_{sc}) &\le 18 k^2 \sqrt{n}
\frac{\norm{\Lambda}_{op}^2}{\norm{\Lambda}_{HS}^2} + C
\frac{\sqrt{\log k}}{k}
\\&= \frac{18 k^2 \sqrt{n}}{\srank \Lambda} + C \frac{\sqrt{\log k}}{k},\end{align*}
and 
\begin{equation*}
  \begin{split}
    \Prob \left[W_1(\mu_{k^{-1/2}M},\rho_{sc}) \ge \frac{18 k^2
        \sqrt{n}}{\srank \Lambda} + C \frac{\sqrt{\log k}}{k} + t\right] & \le \exp\left[-
      \frac{nt^2}{12} \left(2 \frac{\sqrt{n^2-1}}{k}
        \frac{\norm{\Lambda}_{op}}{\norm{\Lambda}_{HS}}\right)^{-2}\right]
    \\
    & \le \exp\left[- \frac{k^2 (\srank \Lambda) t^2}{48 n}\right].
\end{split}
\end{equation*}
\end{thm}

The typical situation of interest in free probability theory is that
of a sequence of $n\times n$ matrices $\Lambda_n$ having a limiting
spectral measure with bounded support. Indeed, suppose that
$\Lambda_n$ is a sequence of traceless diagonal $n\times n$ matrices
which have (as $n\to\infty$) a limiting spectral measure $\mu$ with
compact support, and let $P_k = I_k \oplus 0_{n-k}$ denote orthogonal
projection of $\C^n$ onto the first $k$ coordinates.  Then in the
regime $\frac{k}{n} \to \alpha \in (0,1)$, $\alpha P_k U\Lambda U^* P_k$ has
a limiting spectral measure which can be expressed as
\[
  \mu \boxtimes \bigl(\alpha \delta_\alpha + (1-\alpha) \delta_0\bigr)
  = (1-\alpha) \delta_0 + \alpha \mu^{\boxplus (1/\alpha)}
\]
(see e.g.\ \cite[Section 14]{NiSp}).  Here $\boxtimes$ denotes
multiplicative free convolution, $\delta_x$ denotes the point mass
at $x$, and $\mu^{\boxplus (1/\alpha)}$ denotes an additive free
convolution power of order $1/\alpha$.  It follows that, in the
notation of Theorem \ref{T:DL-semicircle},
\begin{equation}
  \label{E:n-limit}
  \mu_{\alpha^{-1/2} B} \xrightarrow{n \to \infty} \sqrt{\alpha} \mu^{\boxplus (1/\alpha)}.
\end{equation}
The free central limit theorem implies that
\begin{equation}
  \label{E:alpha-limit}
  \sqrt{\alpha} \mu^{\boxplus (1/\alpha)} \xrightarrow{\alpha \to 0} \rho_{sc}.
\end{equation}
As pointed out to the authors by Ion Nechita, this suggests that
$\mu_{k^{-1/2} M} \to \rho_{sc}$ in the setting of Theorem
\ref{T:DL-semicircle} for the entire regime $k = o(n)$ (corresponding
to $\alpha \to 0$).  However, to make this argument rigorous, one would
need to justify interchanging the limits in \eqref{E:n-limit} and
\eqref{E:alpha-limit}.

Theorem \ref{T:DL-semicircle} rigorously shows by other means that
this conclusion is indeed valid for at least part of the regime $k \ll n$. Typically,
one has $\norm{\Lambda_n}_{op} \le C $; the existence of the limiting
spectral measure for $\Lambda_n$ then implies that
$\norm{\Lambda_n}_{HS} \approx \sqrt{n}$, so that
$\srank{\Lambda_n} \approx n$. In this situation, Theorem
\ref{T:DL-semicircle} shows that a semicircular limit holds in
expectation and in probability if $k \to \infty$ and $k \ll n^{1/4}$,
and almost surely (thanks to the Borell--Cantelli lemma) if also
$k \gg \sqrt{\log n}$.  More generally, $\mu_{k^{-1/2}M_n}$ converges
to $\rho_{sc}$ in probability if
\[
\sqrt{\frac{n}{\srank{\Lambda_n}}} \ll k \ll \frac{\sqrt{\srank{\Lambda_n}}}{n^{1/4}}
\]
(note it is possible to choose such $k$ if $\srank{\Lambda_n} \gg
n^{3/4}$) and converges almost surely if
\[
\sqrt{\frac{n \log n}{\srank{\Lambda_n}}} \ll k \ll \frac{\sqrt{\srank{\Lambda_n}}}{n^{1/4}}
\]
(requiring $\srank{\Lambda_n} \gg n^{3/4} \sqrt{\log n}$).

Proofs of Theorem \ref{T:DL-entries-C} and Theorem
\ref{T:DL-semicircle} are given in Section \ref{S:proofs-entries}.

\subsection{Classical invariant ensembles}

Suppose now that $A$ is a random matrix in $M_n^{sa}(\C)$ whose
distribution is invariant under conjugation by unitary matrices; such
classes of random matrices occur frequently in mathematical physics.
The random matrix $A$ has the same distribution as $U \Lambda U^*$,
where $\Lambda$ is a random diagonal matrix with the same eigenvalues
as $A$ and $U$ is a Haar-distributed random matrix in $\Unitary{n}$
which is independent of $\Lambda$.  This observation allows the
marginals of $A$ to be analyzed by applying Theorem \ref{T:DL-T0-C}
conditionally on $\Lambda$. 

\begin{thm}
  \label{T:RL-B-C}
  Let $A$ be a random matrix in $M_n^{sa}(\C)$ whose distribution is
  invariant under unitary conjugation. Let
  $B_1, \dots, B_d \in M_n^{sa}(\C)$ satisfy
  $\tr B_i B_j = \delta_{ij}$ and $\tr B_i = 0$, and define the random
  vector $X \in \R^d$ by $X_i = \tr A B_i$.  Let $g=(g_1,\ldots,g_d)$ be a standard Gaussian random vector in $\R^d$, independent of $A$, and let $\widetilde{A}=A - \frac{1}{n} (\tr A) I_n.$

  Then\begin{equation*}
    W_1\left(X,\frac{\lVert \widetilde{A} \rVert_{HS}}{\sqrt{n^2-1}}g\right) 
    \le \frac{ 8}{\sqrt{n}} \E
    \left(\frac{\lVert \widetilde{A}\rVert_{op}^2}{\lVert
        \widetilde{A} \rVert_{HS}}\right) \sum_{i=1}^d \frac{1}{\srank{B_i}},
  \end{equation*}
  and 
  \begin{equation*}
    W_1\left(\frac{\sqrt{n^2-1}}{\E \lVert \widetilde{A} \rVert_{HS}} X, g\right) 
    \le \frac{8 \sqrt{n}}{\E \lVert \widetilde{A} \rVert_{HS}} \E
    \left(\frac{\lVert \widetilde{A}\rVert_{op}^2}{\lVert
        \widetilde{A} \rVert_{HS}}\right) \sum_{i=1}^d \frac{1}{\srank{B_i}} + \sqrt{d} \frac{\E
    \abs{\lVert \widetilde{A} \rVert_{HS} - \E \lVert \widetilde{A} \rVert_{HS}}}{\E \lVert \widetilde{A} \rVert_{HS}}.
  \end{equation*}

\end{thm}

A widely studied class of random matrices whose distributions are
invariant under unitary conjugation are the \emph{unitarily invariant}
ensembles (sometimes referred to as matrix models); see e.g.\
\cite{Deift,DeGi,PaSh}.  These are random matrices with a density with
respect to Lebesgue measure on $M_n^{sa}(\C)$ proportional to
$\exp(-n \tr V)$ for some function $V:\R \to \R$, where $\tr V(A)$ is
understood in the sense of functional calculus.  Up to the choice of
normalization, the Gaussian Unitary Ensemble is the special case where
$V(x) = x^2$.  The following corollary is an easy consequence of
Theorem \ref{T:RL-B-C} for a large class of potentials $V$; it is
likely that the result holds in greater generality.

\begin{cor}
  \label{T:invariant-C}
  Let $V:\R \to \R$ be twice-differentiable with
  $V''(x) \ge \alpha > 0$ for all $x$, and suppose that $A$ is a
  random matrix in $M_n^{sa}(\C)$ with a density proportional to
  $\exp(- n \tr V)$ with respect to Lebesgue measure on
  $M_n^{sa}(\C)$.
  
  Then, with the notations of Theorem \ref{T:RL-B-C},
  \[
  W_1\left(\frac{\sqrt{n^2-1}}{\E \lVert \widetilde{A} \rVert_{HS}} X, g\right) 
  \le \frac{\kappa}{\sqrt{n}}\sum_{i=1}^d \frac{1}{\srank{B_i}},
  \]
  where $\kappa$ depends only on $\alpha$. 
\end{cor}

Theorem \ref{T:RL-B-C} and Corollary \ref{T:invariant-C} are proved in
Section \ref{S:proofs-invariant}.

\subsection{A probabilistic perspective on the Schur--Horn theorem}

The Schur--Horn theorem characterizes pairs of sequences
$(d_1, \dots, d_n)$ and $(\lambda_1, \dots, \lambda_n)$ of real
numbers which can occur as the diagonal entries and eigenvalues,
respectively, of a real symmetric or complex Hermitian matrix.
Specifically, if $A$ is real symmetric or Hermitian, with diagonal
entries $d_1,\ldots,d_n$ and eigenvalues $\lambda_1,\ldots,\lambda_n$,
then the sequence $(d_i)_{1\le i\le n}$ is \emph{majorized} by
$(\lambda_i)_{1\le i\le n}$ (written
$(d_i)_{i=1}^n \prec (\lambda_i)_{i=1}^n$); that is, $(d_i)_{i=1}^n$
is a convex combination of permutations of $(\lambda_i)_{i=1}^n$.
Conversely, if $(d_i)_{i=1}^n \prec (\lambda_i)_{i=1}^n$, then there
is a real symmetric matrix with diagonal entries $d_1,\ldots,d_n$ and
eigenvalues $\lambda_1,\ldots,\lambda_n$.  See \cite[Section
9.B]{MaOlAr} for further discussion, proofs, and references.

Given a sequence $\lambda_1,\ldots,\lambda_n$ of eigenvalues, the
Schur--Horn theorem identifies exactly which sequences of diagonal
entries are possible.  We now consider this question
probabilistically: given a sequence $\lambda_1,\ldots,\lambda_n$, what
are the diagonal entries of a Hermitian matrix with these eigenvalues
typically like?  This is analogous to the single ring theorem
considered in \cite{Feze,WeFy,GuKrZe,RuVe,GuZe-srt,BeGe}, which can
likewise be viewed as a probabilistic counterpart of the Weyl--Horn
theorem which relates eigenvalues and singular values.  The natural
model of a random Hermitian matrix with the given eigenvalues is of
course $A=U\Lambda U^*$, with $U$ distributed according to Haar
measure on $\Unitary{n}$. (The joint distribution of diagonal entries
was also considered in \cite{SaFySo}; see also \cite[Section
2.2]{FySo}.)

\begin{thm}
  \label{T:Schur-Horn}
  For each $n\in\N$, let
\(
\Lambda_n = \diag(\lambda_1^{(n)}, \dots, \lambda_n^{(n)}),
\)
be a fixed diagonal matrix, and let
$\mu_n$ be the spectral measure of $n^{-1/2} \Lambda_n$:
\[
\mu_n = \frac{1}{n} \sum_{i=1}^n \delta_{n^{-1/2} \lambda_i^{(n)}}.
\]
Suppose that there is a probability measure $\mu$ with mean $m$ and
variance $\sigma^2>0$, such that $W_2(\mu_n,\mu)\to0$.

Let $A_n = U_n \Lambda_n U_n^*$ with $U_n$ Haar-distributed in
$\Unitary{n}$, and   let
$\nu_n$ be the empirical measure of the diagonal entries of
$A_n$:
\[
\nu_n = \frac{1}{n} \sum_{i=1}^n \delta_{a_{ii}^{(n)}}.
\]

If $\bnorm{\widetilde{\Lambda_n}}_{op} = o\left(n\right)$, then
$\nu_n\to \Normal{m}{\sigma^2}$ weakly in probability.  If moreover
$\bnorm{\widetilde{\Lambda_n}}_{op} = o\bigl(\frac{n}{\sqrt{\log
    n}}\bigr)$,
then $\nu_n \to \Normal{m}{\sigma^2}$ weakly almost surely.

Furthermore, if for some constant $K$ and for all $n$,
$\bnorm{\widetilde{\Lambda_n}}_{op} \le K \sqrt{n}$,
then there are constants $\kappa_1, \kappa_2, \kappa_3 > 0$ depending
only on $K$ such that
\begin{align*}
\kappa_1 \sqrt{\log n} &\le \E \Bigl(\max_{1\le i \le n} a_{ii}^{(n)} -
\frac{1}{n} \tr \Lambda_n \Bigr) \\&\le \E \Bigl(\max_{1\le i \le n}\left| a_{ii}^{(n)} -
\frac{1}{n} \tr \Lambda_n\right| \Bigr)\le
\kappa_2 \sqrt{\log n}
\end{align*}
for every $n$, and with probability $1$,
\[
\Bigl(\max_{1\le i \le n} \left|a_{ii}^{(n)}  -
\frac{1}{n} \tr \Lambda_n\right| \Bigr) \le \kappa_3 \sqrt{\log n}
\]
for all sufficiently large $n$.
\end{thm}

Theorem \ref{T:Schur-Horn} is proved in Section
\ref{S:proofs-Schur-Horn}.

\section{Proofs of the main results}
\label{S:proofs-main}

Our main technical tool is the multivariate version of Stein's method
of exchangeable pairs introduced in \cite{ChMe}.  (The method was
extended and refined in \cite{ReRo,EM-Stein-multi,DoSt}, though we
will not particularly make use of those improvements here.)  The following
essentially restates \cite[Theorem 5]{ChMe} (cf.\ also \cite[Theorem
4]{EM-Stein-multi}) and (for the final statement) \cite[Theorem 1]{EM-linear}. 

\begin{thm}
  \label{T:ChMe}
  Suppose that $X$ be a random vector in $\R^d$, and for each $\eps
  \in (0,1)$ there exists a random vector $X_\eps$ such that
  $(X,X_\eps)$ is exchangeable.  Suppose there exist constants $\alpha, \sigma
  > 0$, a function $s(\eps)$, and a random $d \times d$ matrix $F$
  such that
  \begin{enumerate}
  \item $\ds \frac{1}{s(\eps)} \E \bigl[X_\eps - X \big\vert X \bigr]
    \xrightarrow[\eps \to 0]{L_1} - \alpha X$,
  \item $\ds \frac{1}{s(\eps)} \E \bigl[(X_\eps - X) (X_\eps - X)^T
    \big\vert X \bigr] \xrightarrow[\eps \to 0]{L_1}
    2 \alpha \sigma^2 I_d + \E [F \mid X]$, and
  \item for each $\rho > 0$, $\ds \frac{1}{s(\eps)} \E
    \left[\norm{X_\eps - X}^2 \ind{\norm{X_\eps - X}^2 > \rho} \right]
      \xrightarrow{\eps \to 0} 0$.
  \end{enumerate}
  If $g=(g_1,\ldots,g_d)$ is a standard Gaussian random vector, then
  \[
  W_1 (X,\sigma g) \le \frac{1}{2 \alpha \sigma} \E \norm{F}_{HS}.
  \]
  Moreover, if $d = 1$ then
  \[
  d_{TV} (X, \sigma g) \le \frac{1}{\alpha \sigma^2} \E \abs{F}.
  \]
\end{thm}

We will also use these bounds in the equivalent forms
\[
W_1 \left(\frac{1}{\sigma} X, g\right) \le \frac{1}{2
  \alpha \sigma^2} \E \norm{F}_{HS}.
\]
and
\[
d_{TV} \left(\frac{1}{\sigma}X, g\right) \le
\frac{1}{\alpha \sigma^2} \E \abs{F}.
\]

As mentioned above, Theorem \ref{T:ChMe} is one version of Stein's
method for normal approximation.  The basic idea of this version is as
follows. Suppose that the random vector $X$ has ``continuous
symmetries'' which allow one to make a small (parametrized by
$\eps$) random change to $X$ which preserves its distribution.  If
$X$ were exactly Gaussian and this small random change could be made
so that $(X,X_\eps)$ were jointly Gaussian, then we would have
that $X_\eps\overset{d}{=}\sqrt{1-\eps^2}X+\eps Y$ for $X$
and $Y$ independent.  The conditions of the theorem are then
approximate versions of what happens, up to third order, in this
jointly Gaussian case; the random matrix $F$ is thought of as a small
error.  The theorem then says that these conditions are enough to
guarantee approximate Gaussian behavior.

\begin{proof}[Proof of Theorem \ref{T:DL-T0-C}]
  Since $\tr B_i = 0$, $X_i = \tr \widetilde{\Lambda} B_i$.  We may
  therefore assume without loss of generality that $\tr \Lambda = 0$.

  To apply Theorem \ref{T:ChMe}, we must construct an
  appropriate family of random vectors $X_\eps$; our
  construction is an adaptation of one first used by Stein in \cite{Stein}, and 
  later applied in \cite{EM-linear,ChMe}.

  Define
  \begin{equation*}
    \begin{split}
      R_\eps & := \begin{pmatrix} \sqrt{1-\eps^2} & \eps \\ -\eps &
        \sqrt{1-\eps^2} \end{pmatrix} \oplus I_{n-2} \\
      & = I_n + \eps \begin{pmatrix} 0 & 1 \\ -1 & 0 \end{pmatrix} \oplus
      0_{n-2}
      - \dfrac{\eps^2}{2} I_2 \oplus 0_{n-2} + O(\eps^3).
    \end{split}
  \end{equation*}  
  Let $V \in \Unitary{n}$ be Haar-distributed independently of $U$,
  and define $V_\eps := VR_\eps V^*$ and
  $A_\eps := U V_\eps \Lambda V_\eps^* U^*$. Note that
  $(U,UV_\eps)$ is exchangeable by the translation invariance of
  Haar measure.  For each $i$,
  let $(X_\eps)_i = \tr (A_\eps B_i)$.
    
  For notational convenience, define the $n \times 2$ matrix
  $K = [ v_1 v_2 ]$, where $v_i$ are the columns of $V$, and let 
  $Q = K \bigl(\begin{smallmatrix} 0 & 1 \\ -1 &
    0 \end{smallmatrix}\bigr) K^* = v_1 v_2^* - v_2 v_1^*$. Then
  \[
    V_\eps = I_n + \eps Q - \frac{\eps^2}{2} KK^* +
    O(\eps^3)
  \]
  and so (using that $Q^* = -Q$).
  \begin{equation}
    \label{E:A-difference-unitary}
    A_\eps - A =  U \left[ \eps ( Q \Lambda - \Lambda Q )
      - \eps^2 \left( Q \Lambda Q 
        + \frac{1}{2} KK^* \Lambda + \frac{1}{2} \Lambda K K^* \right)
    \right] U^* + O(\eps^3).
  \end{equation}
    
  It is easy to show that $\E Q = 0$ (by conditioning on $v_1$, say) and
  $\E K K^* = \E \left[v_1 v_1^* + v_2 v_2^*\right] = \frac{2}{n}I_n$. Moreover,
  from \cite[Lemma 14]{ChMe} it follows that
  \begin{equation*}
    \label{E:QLambdaQ}
    \E Q \Lambda Q = \frac{2}{(n-1)n(n+1)} \Lambda
  \end{equation*}
  and 
  \begin{equation}
    \label{E:QFG-unitary}
    \begin{split}
      \E [\tr (QF) \tr (QG) ]
      & = \frac{2}{(n-1)n(n+1)} \bigl((\tr F)(\tr G) - n \tr (FG)\bigr)
    \end{split}
  \end{equation}
  for $F, G \in M_n(\C)$.

  Therefore,
  \begin{equation*}\begin{split}
    \label{E:E-A-difference-unitary}
    \E [A_\eps - A \mid U ] &= - \eps^2 U \left[ \frac{2}{(n-1)n(n+1)}
      \Lambda + \frac{2}{n} \Lambda \right] U^* + O(\eps^3)
    = - \frac{2 n \eps^2}{n^2-1} A + O(\eps^3),
  \end{split}\end{equation*}
  and consequently
  \begin{equation*}
    \label{E:E-X-difference-unitary}
    \E [X_\eps - X \mid X ] = - \frac{2 n \eps^2}{n^2-1} X + O(\eps^3),
  \end{equation*}
  so that Theorem \ref{T:ChMe} applies with $s(\eps) = \eps^2$ and
  $\alpha = \frac{2 n}{n^2-1}$.

To identify $\sigma^2$ and $F$ from Theorem \ref{T:ChMe}, we first compute expectations conditional on $U$.  Writing $C_i := U^* B_i U$ and using $\sim$ to denote equality to top order in $\eps$,
  \begin{equation}
    \label{E:Xeps-X-square-unitary}
    \begin{split}
      \E [(X_\eps - X)_i & (X_\eps - X)_j \mid U ] \\
      & \sim
      \eps^2 \E \bigl[ \tr[U(Q\Lambda - \Lambda Q)U^* B_i] 
      \tr[U(Q\Lambda - \Lambda Q)U^* B_j] \mid U \bigr] \\
      & = \eps^2 \E \bigl[ \tr[ (Q\Lambda - \Lambda Q) C_i] 
      \tr[(Q\Lambda - \Lambda Q) C_j] \mid U \bigr].
    \end{split}
  \end{equation}
  By \eqref{E:QFG-unitary},
  \begin{equation}
    \label{E:QLLQ-unitary}
    \begin{split}
      \E \bigl[ \tr[ (Q\Lambda - \Lambda Q) C_i] & \tr[(Q\Lambda -
      \Lambda
      Q) C_j] \mid U \bigr] \\
      & = \E \bigl[ \tr(Q \Lambda C_i) \tr(Q \Lambda C_j) + \tr (Q C_i
      \Lambda) \tr (Q C_j \Lambda) \\
      & \qquad - \tr(Q \Lambda C_i) \tr (Q C_j
      \Lambda) - \tr(Q C_i
      \Lambda)\tr(Q \Lambda C_j) \mid U \bigr] \\
      & = \frac{2}{(n-1)(n+1)} \tr \bigl[ - \Lambda C_i \Lambda C_j - C_i
      \Lambda C_j \Lambda
      + \Lambda C_i C_j \Lambda + C_i \Lambda \Lambda C_j \bigr] \\
      & = \frac{2}{(n-1)(n+1)} \tr \bigl[\Lambda^2 C_i C_j + \Lambda^2 C_j
      C_i - 2 \Lambda C_i \Lambda C_j \bigr] \\
      & = \frac{2}{(n-1)(n+1)} \tr \bigl[ A^2 B_i B_j +
      A^2 B_j B_i - 2 A B_i A B_j\bigr].
    \end{split}
  \end{equation}

Now
\begin{equation}
  \label{E:A^2-unitary}
  \E A^2 = \E U \Lambda^2 U^* = \E \sum_{i=1}^n \lambda_i^2 u_i u_i^* =
  \sum_{i=1}^n \lambda_i^2 \frac{1}{n} I_n =
  \frac{\norm{\Lambda}_{HS}^2}{n} I_n.
\end{equation}
Supposing that $D$ is diagonal,
\begin{equation*}
  \begin{split}
    \E \tr (A D A C) & = \E \tr (U \Lambda U^* D U \Lambda U^* C) \\
    & = \E \sum_{ijk\ell m} u_{ij} \lambda_j \overline{u_{kj}} d_{kk}
    u_{k \ell} \lambda_\ell
    \overline{u_{m \ell}} c_{mi} \\
    & = \sum_{ijk\ell m} \lambda_j \lambda_\ell d_{kk} c_{mi} \E u_{ij} u_{k \ell}
    \overline{u_{kj}} \overline{u_{m\ell}}.
  \end{split}
\end{equation*} 

The latter expectation is nonzero only if $i = m$, and then by \cite[Lemma 14]{ChMe},
\begin{equation}
  \label{E:ADAC}
  \begin{split}
    \E \tr (A D A C) & = \E \tr (U \Lambda U^* D U \Lambda U^* C) \\
    & = \sum_{ikj\ell} \lambda_j \lambda_\ell d_{kk} c_{ii} \E u_{ij}
    u_{k\ell} \overline{u_{i \ell}} \overline{u_{k j}} \\
    & = \frac{1}{(n-1) n (n+1)} \sum_{ikj\ell} \lambda_j \lambda_\ell
    d_{kk} c_{ii} \bigl[ n \delta_{ik} + n \delta_{j\ell} -
    \delta_{ik}\delta_{j\ell}  - 1 \bigr] \\
    & = \frac{1}{(n-1)n(n+1)} \left[n \norm{\Lambda}_{HS}^2 (\tr D) (\tr
      C) - \norm{\Lambda}_{HS}^2 \tr (DC) \right].
  \end{split}
\end{equation} 

If $B$ is Hermitian, we may write $B = YDY^*$ for $Y$ unitary and $D$
diagonal. By the translation invariance of Haar measure, $U$ is equal in distribution to $YU$.  Making this substitution inside the expectation and then using \eqref{E:ADAC} yields
\begin{equation}
  \label{E:tr-ABAC-unitary}
  \begin{split}
    \E \tr (ABAC) & = \E \tr (U \Lambda U^* Y D Y^* U \Lambda U^* C)
    \\
    & = \E \tr (Y U \Lambda U^* D U \Lambda U^* Y^* C) \\
    & = \E \tr (A D A Y^* C Y) \\
    & = \frac{\norm{\Lambda}_{HS}^2}{(n-1)n(n+1)} \left[n (\tr D) (\tr
      Y^* C Y) - \tr (D Y^* C Y) \right] \\
    & = \frac{\norm{\Lambda}_{HS}^2}{(n-1)n(n+1)} \left[n (\tr B) (\tr
      C) - \tr (B C) \right].
  \end{split}
\end{equation}

By \eqref{E:Xeps-X-square-unitary}, \eqref{E:QLLQ-unitary},
\eqref{E:A^2-unitary}, and \eqref{E:tr-ABAC-unitary}, (and the facts
that $\tr (B_i B_j) = \delta_{ij}$ and $\tr B_i = 0$),
\begin{equation*}
  \begin{split}
    \E [(W_\eps - W)_i (W_\eps - W)_j] 
    & \sim \eps^2 \E \bigl( \tr[ (Q\Lambda - \Lambda Q) C_i]
    \tr[(Q\Lambda - \Lambda Q) C_j] \bigr) \\
    & = \frac{2\eps^2}{(n-1)(n+1)} \E \tr \bigl[ A^2
    B_i B_j + A^2 B_j B_i -  2A B_i A B_j \bigr] \\
    & = \frac{2 \norm{\Lambda}_{HS}^2 \eps^2}{(n-1)(n+1)}
    \left( 2 \frac{ \delta_{ij}}{n} + 2 \frac{\delta_{ij}}{(n-1)n(n+1)} \right) \\
    & = \frac{4 n \norm{\Lambda}_{HS}^2 \eps^2}{(n^2-1)^2}
    \delta_{ij}.
  \end{split}
\end{equation*}

Based on this we take
\[
\sigma^2 = \frac{\norm{\Lambda}_{HS}^2}{n^2 - 1} 
\]
and
\begin{equation*}
  \label{E:F-unitary}
  \begin{split}
    F_{ij} & = \frac{2}{n^2-1} \tr \bigl[ A^2 B_i B_j + A^2 B_j
    B_i - 2A B_i A B_j \bigr] - \frac{4 n
      \norm{\Lambda}_{HS}^2}{(n^2-1)^2} \delta_{ij}. \\
    & = \frac{2}{n^2-1} \left[\tr \bigl([A,B_i] [A,B_j]^*\bigr)
    - \E \tr \bigl([A,B_i] [A,B_j]^*\bigr)\right],
  \end{split}
\end{equation*}
where $[A,B]=AB-BA$ is the commutator of the matrices $A$ and $B$.

To apply Theorem \ref{T:ChMe}, we need to estimate
\begin{equation}
  \label{E:F-variance}
\E \norm{F}_{HS} \le \sqrt{\E \norm{F}_{HS}^2} =
\frac{2}{n^2 - 1} \sqrt{\sum_{i,j=1}^d \var \tr \bigl([A,B_i]
  [A,B_j]^*\bigr)}.
\end{equation}

We will estimate the variances in \eqref{E:F-variance} using a
Poincar\'e inequality. As is well known, if $\lambda_1$ is the
smallest nonzero eigenvalue of $-\Delta$ (where $\Delta$ is the
Laplace--Beltrami operator) on a compact Riemannian manifold $\Omega$,
then
\[
\var f(x) \le \frac{1}{\lambda_1} \E \norm{\nabla f(x)}^2
\]
for any smooth function $f:\Omega \to \R$, where $x$ is a random point
distributed according to normalized volume measure on $\Omega$ (see
e.g.\ \cite[Section 3.1]{Ledoux}). An argument in the proof of
\cite[Theorem 3.9]{Voiculescu} shows that if $\Omega = \Unitary{n}$,
then $\lambda_1 = n$.  It follows that if $f:\Unitary{n} \to \R$ is
$L$-Lipschitz with respect to the geodesic distance $d_g$ on
$\Unitary{n}$, then $\var f(U) \le \frac{1}{n}L^2$.  So it suffices
estimate the Lipschitz constant of functions of the form
\[
f(U) = \tr \bigl([U\Lambda U^*,B] [U\Lambda U^*,C]^*\bigr).
\]

Using $U$ and $V$ for the moment to stand for arbitrary matrices in
$\Unitary{n}$ and $A, A', B, C$ to stand for arbitrary matrices, we
observe first that
\begin{equation}
  \label{E:commutator-bound}
  \norm{[A,B]}_{HS} \le 2 \norm{A}_{HS} \norm{B}_{op}
\end{equation}
and hence
\begin{equation}
  \label{E:commutator-difference}
  \norm{[A,B] - [A',B]}_{HS}
  \le 2 \norm{B}_{op} \norm{A - A'}_{HS}.
\end{equation}
Also,
\begin{equation}
  \label{E:U-V}
  \norm{U\Lambda U^* - V\Lambda V^*}_{HS} = \norm{(U-V) \Lambda U^* + V \Lambda
    (U-V)^*}_{HS}
  \le 2 \norm{\Lambda}_{op} \norm{U-V}_{HS}.
\end{equation}
Now writing $A = U\Lambda U^*$ and $A' = V\Lambda V^*$, it follows
from the Cauchy--Schwarz inequality, \eqref{E:commutator-difference},
\eqref{E:commutator-bound}, and \eqref{E:U-V} that
\begin{equation*}
  \begin{split}
    \abs{f(U) - f(V)} 
    & = \abs{\tr \left([A,B] \bigl([A, C] - [A', C]\bigr)^*\right)
      + \tr \left(\bigl([A,B] - [A',B]\bigr)[A', C]^*\right)} \\
    & \le \norm{[A,B]}_{HS} \norm{[A,C] - [A', C]}_{HS} + \norm{[A',C]}_{HS} \norm{[A,B] - [A', B]}_{HS} \\
    & \le 16 \norm{B}_{op} \norm{C}_{op} \norm{\Lambda}_{op}^2
    \norm{U-V}_{HS} \\
    & \le 16 \norm{B}_{op} \norm{C}_{op} \norm{\Lambda}_{op}^2 d_g(U,V), \\
  \end{split}
\end{equation*}
where the last estimate follows since $\norm{U-V}_{HS} \le d_g(U,V)$ (see e.g.\ \cite[Lemma
1.3]{EM-book}).

The Poincar\'e inequality now implies that
\[
  \var \tr \bigl([A,B_i] [A,B_j]^*\bigr) \le \frac{16^2}{n}
  \norm{B_i}_{op}^2 \norm{B_j}_{op}^2 \norm{\Lambda}_{op}^4,
\]
and so by \eqref{E:F-variance},
\begin{equation}
  \label{E:F-norm}
  \E \norm{F}_{HS} \le \frac{32 \norm{\Lambda}_{op}^2}{\sqrt{n}(n^2 -
    1)} \sum_{i=1}^d \norm{B_i}_{op}^2.
\end{equation}
The theorem now follows directly from Theorem \ref{T:ChMe}.
\end{proof}

The rather soft approach to the bound in
\eqref{E:F-norm} used above, based on Poincar\'e inequalities, yields
an optimal bound in general.  However, it is in principle possible (though
unwieldy) to compute the variances appearing in \eqref{E:F-variance}
explicitly using the Weingarten calculus, and obtain a better result
in certain cases.  We discuss this point further following the proof
of Theorem \ref{T:R1-T0-C} below.

The proof of Theorem \ref{T:DL-T0-R} is a straightforward modification
of the proof above. The required mixed moments of entries of random
orthogonal matrices can also be found in \cite{ChMe}.  For the
Poincar\'e inequality estimate, one must condition on the coset of
$\SOrthogonal{n}$ within $\Orthogonal{n}$; for similar arguments, see,
e.g., \cite{MeMe-concentration}.  The required spectral gap estimate
on $\SOrthogonal{n}$ can be found in \cite{Saloff-Coste}.

\begin{proof}[Proof of Corollary \ref{T:DL-B-C}]
As in the statement of the corollary, let $\Lambda\in M_n(\R)$ be
diagonal and let $B_1, \dots, B_d\in M_n^{sa}(\C)$.  The random matrix
$A$ is defined by $A=U\Lambda U^*$, where $U$ is Haar-distributed in
$\Unitary{n}$, and for each $j$, $X_j=\tr(AB_j)$.

Recall that for
any $B\in M_n(\C)$, we denote by $\widetilde{B}$ the traceless
recentering of $B$:
\[\widetilde{B}=B-\frac{1}{n}(\tr(B))I_n.\]
Note that
\(\widetilde{A} = A - \frac{1}{n} (\tr
  \Lambda) I_n  = U \widetilde{\Lambda} U^*. \)
Also, for each $j$,

\[X_j = \tr A B_j = \tr \bigl(\widetilde{A} +
  \frac{1}{n} (\tr \Lambda) I_n \bigr) \bigl(\widetilde{B}_j +
  \frac{1}{n} (\tr B_j)I_n\bigr) = \tr \widetilde{A}
  \widetilde{B}_j + \frac{1}{n} (\tr \Lambda) (\tr B_j).\]

Recall that the matrix
 $\Sigma$ is
  given by
  \[
    \Sigma_{ij} = \tr \widetilde{B}_i \widetilde{B}_j = \tr B_i B_j -
    \frac{1}{n} \tr B_i \tr B_j;
  \]
  it is nonnegative definite, and positive definite if the $\widetilde{B}_i$ are linearly independent.

If we define $C_i = \sum_{j=1}^d [\Sigma^{-1/2}]_{ij}
\widetilde{B}_j$, then
  \[
    \tr C_i C_j = \sum_{\ell, m=1} [\Sigma^{-1/2}]_{i \ell} \bigl(\tr
    \widetilde{B}_\ell \widetilde{B}_m \bigr) [\Sigma^{-1/2}]_{mj}
    = [\Sigma^{-1/2} \Sigma \Sigma^{-1/2}]_{ij} = \delta_{ij}
  \]
  and
  \[
    \sum_{j=1}^d [\Sigma^{1/2}]_{ij} C_j = \sum_{j, \ell=1}^d
    [\Sigma^{1/2}]_{ij} [\Sigma^{-1/2}]_{j\ell} \widetilde{B}_\ell = \widetilde{B}_i. 
  \]

Now let $W_j = \tr \widetilde{A} C_j$ and
  $v_j = \frac{1}{n} (\tr \Lambda) (\tr B_j)$, so that
$X = \Sigma^{1/2} W + v$.

Theorem \ref{T:DL-T0-C} applied to $W$ gives that
\[
  W_1\left(\frac{\sqrt{n^2-1}}{\bnorm{\widetilde{\Lambda}}_{HS}}W,g\right) \le
  8d\sqrt{n}\frac{\bnorm{\widetilde{\Lambda}}_{op}^2}{\bnorm{\widetilde{\Lambda}}_{HS}^2},
\]
using the trivial estimate
$\srank{C_j} \ge 1$.  Note that for a matrix $M$, multiplication by
$M$ is $\norm{M}_{op}$-Lipschitz, and so
\[
W_1(MX,MY) = \sup_{\abs{f}_L\le 1} \abs{\E f(MX)-\E f(MY)} \le
\norm{M}_{op}W_1(X,Y).
\]
It thus follows from above that 
\[
  W_1\left(X,\frac{
        \bnorm{\widetilde{\Lambda}}_{HS}}{\sqrt{ n^2-1}}
    \Sigma^{1/2} g + v\right) \le
  \frac{8d\norm{\Sigma^{1/2}}_{op} \bnorm{\widetilde{\Lambda}}_{op}^2 
  }{\bnorm{\widetilde{\Lambda}}_{HS} \sqrt{n-1}}. 
  \qedhere
\]  
\end{proof}

For the proof of Theorem \ref{T:R1-T0-C}, we will make use of the fact
that, when restricted to the sphere, traceless Hermitian matrices
acting as bilinear forms on Euclidean space define eigenfunctions of
the Laplacian.  This fact is used in conjunction with the following
theorem from \cite{EM-Stein-multi}.  We note that this theorem is also proved via Theorem \ref{T:ChMe}, so that ultimately, the proofs of Theorems \ref{T:DL-T0-C} and \ref{T:R1-T0-C} rely on the same underlying ideas.

\begin{thm}
  \label{T:eigenfunctions}
  Let $\Omega$ be a compact Riemannian manifold.  Let $f_1,\ldots,f_d$ be
  eigenfunctions of the Laplace-Beltrami operator on $\Omega$, with
  eigenvalues $-\mu_1,\ldots,-\mu_d$, and suppose that the $f_i$ are
  orthonormal in $L_2(\Omega)$ (with the volume measure normalized to have
  total mass 1).  If $Y$ is distributed uniformly (i.e., according to
  volume measure) on $\Omega$ and $X=(f_1(Y),\ldots,f_d(Y))$, then 
  \[
    W_1(X,g)\le \left(\max_{1\le i\le
        d}\frac{1}{\mu_i}\right)\E\sqrt{\sum_{i,j=1}^d\left[\inprod{\nabla
          f_i(Y)}{\nabla f_j(Y)}-\E\inprod{\nabla f_i(Y)}{\nabla
          f_j(Y)}\right]^2}.
  \]
\end{thm}

Making use of the theorem involves integrating various polynomials
over the complex sphere.  The proof of the following lemma is a standard exercise; see, e.g.,
Section 2.7 of \cite{Fol}.
\begin{lemma}\label{T:spherical-integral-formula-1}
Let $Z=(Z_1,\ldots,Z_n)$ be uniformly distributed on the complex unit
sphere $\{z\in\C^n:\sum_{j=1}^n\abs{z_j}^2=1\}$.  Let 
$\alpha_1,\ldots,\alpha_n\in\R_+$, and define
$\beta_j:=\frac{\alpha_j}{2}+1$ and $\beta=\sum_{j=1}^n\beta_j$.  Then
\[\E\left[\abs{Z_1}^{\alpha_1}\cdots\abs{Z_n}^{\alpha_n}\right]=\frac{\Gamma(\beta_1)\cdots\Gamma(\beta_n)\Gamma(n)}{\Gamma(\beta)}.\]
\end{lemma}

The following compact expression for the mixed moments puts Lemma
\ref{T:spherical-integral-formula-1} into a form better suited to our purposes.
\begin{prop}\label{T:spherical-integral-formula-2}
Let $Z=(Z_1,\ldots,Z_n)$ be uniformly distributed on the complex unit
sphere.  The only non-zero mixed moments of the entries of $Z$ and their
conjugates are those in which each entry appears the same time as its
conjugate; these moments are given by the following formula:
\[
\E\big[Z_{i_1}\ldots
Z_{i_k}\overline{Z}_{j_1}\ldots\overline{Z}_{j_k}\big]=\frac{1}{n(n+1)\cdots(n+k-1)}\sum_{\pi\in
  S_k}\delta_{i_1j_{\pi(1)}}\cdots\delta_{i_kj_{\pi(k)}}.
\]
\end{prop}

\begin{proof}
If any entry does not appear the same number of times as its
conjugate, then by the invariance of Haar measure under the
multiplication of a single coordinate by any unit modulus complex
number, the expectation must vanish.  It follows, then, that the
expectation in the statement of the Proposition must in fact be a
mixed absolute moment as in the statement of Lemma
\ref{T:spherical-integral-formula-1}, with $\beta=k+n$, and by the
functional equation $\Gamma(x+1)=x\Gamma(x)$, it follows that
$\frac{\Gamma(n)}{\Gamma(\beta)}=\frac{1}{(n+k-1)\cdots n}$.

Next, observe that for a given $j$,
$\Gamma(\beta_j)=\left(\frac{\alpha_j}{2}\right)!,$ (since $\alpha_j$
is necessarily even), and thus $\Gamma(\beta_j)$ is exactly the number of
matchings of the $\frac{\alpha_j}{2}$ of the $Z_{i_\ell}$ with
$i_\ell=j$ with those $\overline{Z}_{j_\ell}$ with $j_\ell=j$.  It
follows that $\Gamma(\beta_1)\ldots\Gamma(\beta_n)$ is the number of
matchings of the $Z_{i_\ell}$ with the $Z_{j_\ell}$ so that each $Z_j$
is always matched with $\overline{Z}_j$.  This is exactly the
expression given by the sum over permutations formula in the statement
of the Proposition. 
\end{proof}

\begin{proof}[Proof of Theorem \ref{T:R1-T0-C}]
  First observe that
  $\norm{B_j}_4^4 \le \norm{B_j}_{op}^2 \norm{B_j}_{HS}^2 =
  \norm{B_j}_{op}^2$, and so the second bound of the Theorem follows
  immediately from the first.

  Turning to the proof of the first bound, first note that for
  $z\in\C^n$, $\inprod{B_jz}{z}$ is necessarily real, since $B_j$ is
  Hermitian. For a Hermitian matrix $B$, write $B=B_r+iB_i$ with $B_r$
  real and symmetric, and $B_i$ real and anti-symmetric.  Then letting
  $z=x+iy$,
\[
\inprod{Bz}{z} = \inprod{B_rx}{x} + \inprod{B_ry}{y} -
\inprod{B_iy}{x} + \inprod{B_ix}{y}.
\]
That is, $\inprod{Bz}{z}$ viewed as a function on $\sphere{2n-1}$
(associating $z=x+iy$ with $(x,y)\in\R^{2n}$)
corresponds to the bilinear form with symmetric traceless matrix

\begin{equation}\label{E:real-matrix}
  \tilde{B}=\begin{bmatrix}B_r&-B_i\\B_i&B_r\end{bmatrix},\end{equation}
which is an
eigenfunction of the spherical Laplacian with eigenvalue $-4n$.  (See
\cite{EM-eigenfunctions} for details on this statement and facts about gradients
needed below).

While it is necessary to view $\inprod{Bz}{z}$ as a function on
$\sphere{2n-1}$ in order to apply Theorem \ref{T:eigenfunctions},
evaluating the integrals needed is generally simpler in the complex
setting; this is justified since the push-forward of uniform measure
on the complex unit sphere to $\sphere{2n-1}$ is again the uniform
measure.

By Proposition \ref{T:spherical-integral-formula-2}, for $B$ Hermitian
and traceless,
\[\E\inprod{BZ}{Z}=\sum_{j,k=1}^nb_{jk}\E\left[Z_k\overline{Z}_j\right]=\frac{1}{n}\tr(B)=0,\]
and for $B,C$ Hermitian and traceless,
\begin{align*}\E\left(\inprod{BZ}{Z}\inprod{CZ}{Z}\right)&=\sum_{j,k,\ell,m}b_{jk}c_{\ell
  m}\E\left[Z_kZ_m\overline{Z}_j\overline{Z}_\ell\right]\\&=\frac{1}{n(n+1)}\left[\sum_{j,\ell}\left(b_{jj}c_{\ell\ell}+b_{j\ell}c_{\ell j}\right)\right]=\frac{\tr(BC)}{n(n+1)},\end{align*}
and so if $f_j(z)=\sqrt{n(n+1)}\inprod{B_iz}{z}$, then
$f_1(Z),\ldots,f_d(Z)$ are orthonormal eigenfunctions of the Laplacian
on $\sphere{2n-1}$.

Now, the gradient $\nabla f_j$ appearing in Theorem
\ref{T:eigenfunctions} is the gradient defined by the Riemannian
metric; in this case, it is the spherical gradient, which is given by
\[\nabla_{S^{2n-1}}f(z)=\nabla_{\R^{2n}}f(z)-\inprod{z}{\nabla_{\R^{2n}}f(z)}z.\]
(Abusing notation, we are treating $z$ as a vector in
$\R^{2n}$: $(z_1,\ldots,z_{2n})=(x_1,\ldots,x_n,y_1,\ldots,y_n)$.)
For $f_j$ defined as above and $\tilde{B}$ as in
\eqref{E:real-matrix},
\[\nabla_{\sphere{2n-1}}f_j(z)=\sqrt{n(n+1)}\left[2\tilde{B}_jz-2\inprod{\tilde{B}_jz}{z}z\right]=2\sqrt{n(n+1)}\tilde{B}_jz-2zf_j(z),\]
and so 
\[\inprod{\nabla_{\sphere{2n-1}}f_j(z)}{\nabla_{\sphere{2n-1}}f_k(z)}=4n(n+1)\inprod{\tilde{B}_jz}{\tilde{B}_kz}-4f_j(z)f_k(z)\]

Now, 
\(\tilde{B}_jz=(\Re(B_jz),\Im(B_jz)),\)
and so 
\[\inprod{\tilde{B}_jz}{\tilde{B}_kz}=\inprod{\Re(B_jz)}{\Re(B_kz)}+\inprod{\Im(B_jz)}{\Im(B_kz)}=\Re\left(\inprod{B_jz}{B_kz}\right).\]
That is,
\begin{equation}\label{E:grad-dot-product}\inprod{\nabla_{\sphere{2n-1}}f_j(z)}{\nabla_{\sphere{2n-1}}f_k(z)}=4n(n+1)
\Re\left(\inprod{B_jz}{B_kz}\right) -4f_j(z)f_k(z).\end{equation}
Taking the expectation using Proposition
\ref{T:spherical-integral-formula-2},
\[\E\inprod{B_jZ}{B_kZ}=\sum_{\ell,p,q}[B_j]_{\ell
  p}\overline{[B_k]_{\ell
    q}}\E\left[Z_p\overline{Z}_q\right]=\frac{1}{n}\tr(B_jB_k^*)=\frac{\delta_{jk}}{n},\]
and by the orthonormality of the $f_j$, this means that
\[\E\inprod{\nabla_{\sphere{2n-1}}f_j(z)}{\nabla_{\sphere{2n-1}}f_k(z)}=4n\delta_{jk}.\]

We now estimate the variance of this expression.  
For notational convenience, write $B_j=A$ and $B_k=C$.  Then
\begin{align*}\E\inprod{Az}{Cz}^2
  &=\sum_{\ell,m,p,q,r,s}a_{\ell
  p}a_{m r}\overline{c_{\ell q}}\overline{c_{m
    s}}\E\left[Z_pZ_r\overline{Z}_q\overline{Z}_s\right]\\
  &=\frac{1}{n(n+1)}\sum_{\ell,m,p,r}\left(a_{\ell
  p}a_{m r}\overline{c_{\ell p}}\overline{c_{m r}}+a_{\ell
  p}a_{m r}\overline{c_{\ell r}}\overline{c_{m p}}\right) \\
  & =\frac{1}{n(n+1)}\left(\tr(AC^*)^2+\tr\left((AC^*)^2\right)\right)
\end{align*}
and 
\begin{align*}
\E\abs{\inprod{Az}{Cz}}^2
&=\sum_{\ell,m,p,q,r,s}a_{\ell
  p}\overline{a_{m r}}\overline{c_{\ell q}}c_{m
  s}\E\left[Z_p\overline{Z}_r\overline{Z}_qZ_s\right]\\
&=\frac{1}{n(n+1)}\sum_{\ell,m,p,s}\left(a_{\ell
  p}\overline{a_{m p}}\overline{c_{\ell s}}c_{m s}+a_{\ell
  p}\overline{a_{m s}}\overline{c_{\ell p}}c_{m s}\right)\\
&=\frac{1}{n(n+1)}\left(\tr(AA^*CC^*)+\tr(AC^*)^2\right),\end{align*}
and so 
\begin{align*}16n^2(n+1)^2\E\left[\Re\left(\inprod{B_jz}{B_kz}\right)\right]^2&=8n^2(n+1)^2\E\left(\Re\left[\inprod{B_jz}{B_kz}^2+\abs{\inprod{B_jz}{B_kz}}^2\right]\right)\\&=8n(n+1)\Re\left[2\delta_{jk}+\tr\left((B_jB_k)^2\right)+\tr(B_j^2B_k^2)\right].\end{align*}
By H\"older's inequality for unitarily invariant norms (see \cite[Corollary IV.2.6]{Bhatia}) and the Cauchy--Schwarz inequality,
\[
\tr\left((B_jB_k)^2\right)\le\norm{B_jB_k}_{HS}\norm{B_kB_j}_{HS}\le\norm{B_j^2}_{HS}\norm{B_k^2}_{HS}=\norm{B_j}_4^{2}\norm{B_k}_4^{2}.
\]
Similarly,
\[
\tr(B_j^2B_k^2)\le\norm{B_j^2}_{HS}\norm{B_k^2}_{HS}=\norm{B_j}_4^{2}\norm{B_k}_4^{2},
\]
and so
\begin{align*}
  16n^2(n+1)^2&\E\left[\Re\left(\inprod{B_jz}{B_kz}\right)\right]^2\le
                16n(n+1)\left[\delta_{jk}+
                \norm{B_j}_4^{2}\norm{B_k}_4^{2}\right].
\end{align*}

Next, regarding $\pi\in S_3$ as a bijection from $\{p,r,t\}$ to
itself, using Proposition \ref{T:spherical-integral-formula-2} and
assuming $B$ and $C$ are Hermitian and traceless,
\begin{align*}
\E\left[\inprod{Bz}{Cz}\inprod{Bz}{z}\inprod{Cz}{z}\right]&=\sum_{\ell,p,q,r,s,t,u}b_{\ell
  p}b_{rs}\overline{c_{\ell  q}}c_{tu}\E\left[Z_{p}Z_rZ_t\overline{Z_q}\overline{Z_s}\overline{Z_u}\right]\\&=\frac{1}{n(n+1)(n+2)}\sum_{\ell,p,r,t}\sum_{\pi\in S^3}b_{\ell p}b_{r\pi(p)}\overline{c_{\ell\pi(r)}}c_{t\pi(t)}\\&=\frac{1}{n(n+1)(n+2)}\left[\tr(BC^*)^2+\tr(BC^TB^TC)+\tr(BB^TC^TC)\right],
\end{align*}
and so 
\begin{align*}
32n(n+1)\E\left[\Re\left(\inprod{B_jz}{B_kz}\right)f_j(Z)f_k(Z)\right]&=\frac{32n(n+1)}{n+2}\Re\left[\delta_{jk}+\tr(B_jB_k^TB_j^TB_k)+\tr(B_jB_j^TB_kB_k^T)\right].
\end{align*}
Bounding the traces as above, we have 
\[\abs{32n(n+1)\E\left[\Re\left(\inprod{B_jz}{B_kz}\right)f_j(Z)f_k(Z)\right]}\le 32 n[\delta_{jk}+2\norm{B_j}_4^{2}\norm{B_k}_4^{2}].\]

Lastly, letting $B$ and $C$ be traceless Hermitian matrices and using
Proposition \ref{T:spherical-integral-formula-2} as above,
\begin{align*}
\E&\left[\inprod{BZ}{Z}^2\inprod{CZ}{Z}^2\right]\\&\qquad=\sum_{\ell,m,p,q,r,s,t,u}b_{\ell  m}b_{p q}c_{r  s}c_{tu}\E\left[Z_mZ_qZ_sZ_u\overline{Z_\ell  Z_pZ_rZ_t}\right]\\&\qquad=\frac{1}{n(n+1)(n+2)(n+3)}\sum_{\ell,p,r,t}\sum_{\pi\in S_4}b_{\ell   \pi(\ell)}b_{p \pi(p)}c_{r  \pi(r)}c_{t\pi(t)}\\&\qquad=\frac{1}{n(n+1)(n+2)(n+3)}\left[\tr(B^2)\tr(C^2)+4\tr(B^2C^2)+2\tr(BC)^2+2\tr\left((BC)^2\right)\right].
\end{align*}
It follows that
\[16\E
f_j^2(Z)f_k^2(Z)=\frac{16n(n+1)}{(n+2)(n+3)}\left[1+2\delta_{jk}+4\tr(B_j^2B_k^2)+2\tr\left((B_jB_k)^2\right)\right],\]
and thus
\[\abs{16\E
f_j^2(Z)f_k^2(Z)}\le 16\left[1+2\delta_{jk}+6\norm{B_j}_4^{2}\norm{B_k}_4^{2}\right]\]
All together then, there is an absolute constant $C$ such that
\[
\sum_{j,k=1}^d\var\inprod{\nabla_{\sphere{2n-1}}f_j(z)}{\nabla_{\sphere{2n-1}}f_k(z)}\le
C\left[n^2\left(\sum_{j=1}^d\norm{B_j}_4^{2}\right)^2+nd\right].
\]
For each $j$,
$\norm{B_j}_4^2 \ge n^{-1/2} \norm{B_j}_{HS} = n^{-1/2}$, so the
second term in the last estimate above is bounded by the first.  This
completes the proof.
\end{proof}

The improvements in Theorem \ref{T:R1-T0-C}, relative to applying
Theorem \ref{T:DL-T0-C} in the case that $\operatorname{rank} \Lambda
= 1$, result from more explicitly computing a variance, which here
involves integrating polynomials of degree $8$ on the complex sphere.  
It appears to be possible to improve Theorem \ref{T:DL-T0-C}, in such
a way that Theorem \ref{T:R1-T0-C} can be recovered as a special case,
by integrating polynomials of degree $8$ on the unitary group at a
corresponding step, using the Weingarten calculus.  (This is
circumvented in the proof of Theorem \ref{T:DL-T0-C} above using a
Poincar\'e inequality and Lipschitz estimate.) Details will appear in
future work.

\section{Random quantum states: proof of Theorem
  \ref{T:induced-states}}
\label{S:proofs-quantum}

\begin{proof}[Proof of Theorem \ref{T:induced-states}]
First observe that 
\[
X_j = \tr\big(\tr_1(ZZ^*)B_j\big)
 = \tr\big(ZZ^*(B_j\otimes I_s)\big)
 = \inprod{(B_j\otimes I_s)Z}{Z}.
\]
Now, 
\[
\inprod{B_j\otimes I_s}{B_k\otimes I_s}
=\tr(B_jB_k\otimes I_s)
=s\tr(B_jB_k)
=s\delta_{jk},
\]
and so if
$Y_j=\frac{1}{\sqrt{s}}\inprod{\left(B_j\otimes I_s\right)Z}{Z}$, then
Theorem \ref{T:R1-T0-C} applies, and
\[
W_1(\sqrt{ns(ns+1)}Y,g)\le C
\sum_{j=1}^d\norm{\frac{1}{\sqrt{s}}\left(B_j\otimes
    I_s\right)}_4^{2}.
\]
Since 
\[
\norm{\frac{1}{\sqrt{s}}\left(B_j\otimes I_s\right)}_4^2
=\frac{1}{s}\sqrt{\tr\left((B_j\otimes I_s)^4\right)}
=\frac{1}{\sqrt{s}}\sqrt{\tr\left(B_j^4\right)},
\]
this completes the proof.
\end{proof}

\section{Joint distributions of entries: proofs of Theorem
  \ref{T:DL-entries-C} and Theorem \ref{T:DL-semicircle}}
\label{S:proofs-entries}

\begin{proof}[Proof of Theorem \ref{T:DL-entries-C}]
  Let $E_{jk} \in M_n(\R)$ denote the matrix with $j,k$ entry equal to
  $1$ and all other entries $0$. We apply Corollary \ref{T:DL-B-C}, choosing all of the matrices $B_i$ to be elements of the standard basis of the space of Hermitian matrices: let $r\le d$ be the number of $B_i$ of the form $B_{jj}^D = E_{jj}$, with the remaining $d-r$ having the form 
  \[
  B_{jk}^R = \frac{1}{\sqrt{2}} (E_{jk} + E_{kj})
  \qquad \text{or} \qquad
  B_{jk}^I = \frac{i}{\sqrt{2}} (E_{jk} - E_{kj})
  \]
  for $1\le j<k\le n$. Thus $\tr (A B_{jj}^D) = a_{jj}$,
  $\tr (A B_{jk}^R) = \sqrt{2}\Re (a_{jk})$, and
  $\tr (A B_{jk}^I) = \sqrt{2}\Im (a_{jk})$. Observe that $B_{jk}^R$
  and $B_{jk}^I$ are traceless, while
  $\widetilde{B_{jj}^D} = B_{jj}^D - \frac{1}{n}I_n$.  It follows that
  in the setting of Corollary \ref{T:DL-B-C}, if we order the
  coefficient matrices so that those of the form $B_{jj}^D$ are listed
  first, then $\Sigma = I_d - \frac{1}{n} J_r$, where
  $J_r \in M_d(\R)$ consists of an $r\times r$ block of $1$s in the
  upper-left corner, with all other entries $0$.  Corollary \ref{T:DL-B-C} then
  implies that
  \begin{equation*}
    \begin{split}
      W_1\left(\frac{\sqrt{n^2-1}}{\norm{\Lambda}_{HS}} X, g \right) &
      \le W_1\left(\frac{\sqrt{n^2-1}}{\norm{\Lambda}_{HS}} X,
        \Sigma^{1/2} g \right)
      + W_1 \left(\Sigma^{1/2} g, g\right) \\
      & \le 8 d \sqrt{n} \bnorm{\Sigma^{1/2}}_{op}
      \frac{\norm{\Lambda}_{op}^2}{\norm{\Lambda}_{HS}^2} 
      + W_1 \left(\Sigma^{1/2} g, g\right).
    \end{split}
  \end{equation*}
  Now 
  \[
  W_1 \left(\Sigma^{1/2} g, g\right) \le \sup_{\abs{f}_L \le 1}
  \abs{\E f \bigl(\Sigma^{1/2} g\bigr) - \E f(g)}
  \le \norm{\Sigma^{1/2} - I_d}_{op}
  \E \norm{g} \le \sqrt{d} \norm{\Sigma^{1/2} - I_d}_{op}.
  \]
  From the description given above, it is immediate that $\Sigma$ has
  eigenvalues $1$ (with multiplicity $d-1$) and $1 - \frac{r}{n}$
  (with multiplicity $1$), so that $\norm{\Sigma^{1/2}}_{op} = 1$ and
  $\norm{\Sigma^{1/2} - I_d}_{op} = 1 - \sqrt{1-\frac{r}{n}} \le
  \frac{r}{n}$.  We thus obtain
  \begin{equation*}
      W_1\left(\frac{\sqrt{n^2-1}}{\norm{\Lambda}_{HS}} X, g \right)
      \le 8 d \sqrt{n}
      \frac{\norm{\Lambda}_{op}^2}{\norm{\Lambda}_{HS}^2} 
      + \frac{\sqrt{d}r}{n}.
  \end{equation*}
  The stated bound now follows, since $\norm{\Lambda}_{HS}^2 \le n
  \norm{\Lambda}_{op}^2$ and $r \le \max\{n,d\}\le \sqrt{dn}$.  
\end{proof}

\begin{proof}[Proof of Theorem \ref{T:DL-semicircle}]
  For $G$ a $k\times k$ GUE matrix, it was proved by Dallaporta
  \cite{Dal} that there is a constant $C$, independent of $k$, such that
\begin{equation}\label{E:GUE-sc}
\E W_1(\mu_{k^{-1/2} G},\rho_{sc}) \le C \frac{\sqrt{\log k}}{k}.
\end{equation}
The Hoffman--Wielandt inequality \cite[Theorem VI.4.1]{Bhatia} implies
that that $C \mapsto \mu_C$ is $\frac{1}{\sqrt{k}}$-Lipschitz for $k\times
k$ normal matrices (taking $W_1$ as the metric on probability measures
and the Hilbert-Schmidt distance on matrices),
so for any coupling of normal random matrices $M_1$ and $M_2$,
\begin{equation*}
    \E W_1(\mu_{M_1}, \mu_{M_2}) \le\frac{1}{\sqrt{k}}\E\norm{M_1-M_2}_{HS},
\end{equation*}
and by taking infimum over couplings,
\[\E W_1(\mu_{M_1}, \mu_{M_2}) \le\frac{1}{\sqrt{k}}W_1(M_1,M_2).\]
Writing $M = \frac{\sqrt{n^2-1}}{\norm{\Lambda}_{HS}} B$, it follows
from Theorem \ref{T:submatrix} that
\[
\E W_1(\mu_{k^{-1/2}M}, \mu_{k^{-1/2}G}) \le 18 k^2 \sqrt{n}
\frac{\norm{\Lambda}_{op}^2}{\norm{\Lambda}_{HS}^2}.
\]
Combining this with the estimate \eqref{E:GUE-sc} yields
\[
\E W_1 (\mu_{k^{-1/2}M}, \rho_{sc}) \le 18 k^2 \sqrt{n}
\frac{\norm{\Lambda}_{op}^2}{\norm{\Lambda}_{HS}^2} + C
\frac{\sqrt{\log k}}{k}
= \frac{18 k^2 \sqrt{n}}{\srank \Lambda} + C \frac{\sqrt{\log k}}{k},
\]
which is the first statement of the Theorem.

To prove the second statement, consider the mapping $U \mapsto A
\mapsto B$, where  $A=U\Lambda  U^*$ and $B$ is  the upper-left
$k\times k$ submatrix of $A$.  Observe that 
\begin{equation}
  \begin{split}
    \Bigl\vert\norm{U_1\Lambda U_1^*}_{HS}-\norm{U_2\Lambda
      U_2^*}_{HS}\Bigr\vert
    & \le \norm{(U_1-U_2)\Lambda
      U_1^*}_{HS}+\norm{U_2\Lambda(U_1^*-U_2^*)}_{HS} \\
    & \le 2\norm{\Lambda}_{op}\norm{U_1-U_2}_{HS},
  \end{split}
\end{equation}
and
$A\mapsto B$ is a projection, so $B$ is a
$2\norm{\Lambda}_{op}$-Lipschitz function of $U$.  It follows that
$W_1(\mu_{k^{-1/2}M},\rho_{sc})$ is a
$2 \norm{\Lambda}_{op}\frac{\sqrt{n^2-1}}{\norm{\Lambda}_{HS}}
\frac{1}{k}$-Lipschitz
function of $U$.  Lipschitz functions on $\Unitary{n}$ satisfy the
sub-Gaussian concentration inequality
\begin{equation}
  \label{E:unitary-concentration}
\Prob[F(U)\ge\E
F(U)+t]\le\exp\left[-\frac{nt^2}{12\abs{F}_L^2}\right],
\end{equation}
(see \cite[Corollary 17]{MeMe-powers}), and so
\begin{equation*}
  \begin{split}
    \Prob \left[W_1(\mu_{k^{-1/2}M},\rho_{sc}) \ge \E
      W_1(\mu_{k^{-1/2}M},\rho_{sc}) + t\right] & \le \exp\left[-
      \frac{nt^2}{12} \left(2 \frac{\sqrt{n^2-1}}{k}
        \frac{\norm{\Lambda}_{op}}{\norm{\Lambda}_{HS}}\right)^{-2}\right]
    \\
    & \le \exp\left[- \frac{k^2 (\srank \Lambda) t^2}{48 n}\right].
    \qedhere
\end{split}
\end{equation*}
\end{proof}

\section{Invariant ensembles: proof of Theorem \ref{T:RL-B-C} and
  Corollary \ref{T:invariant-C}}
\label{S:proofs-invariant}

\begin{proof}[Proof of Theorem \ref{T:RL-B-C}]
  As discussed prior to the statement of the theorem, the random
  matrix $A$ has the same distribution as $U \Lambda U^*$, where
  $\Lambda$ is a real diagonal random matrix with the same eigenvalues
  as $A$ and $U$ is Haar-distributed in the unitary group, independent
  from $\Lambda$.

  Observe that since $\tr B_i = 0$, $\tr AB_i = \tr \widetilde{A}B_i$,
  and that
  $\bnorm{ \widetilde{A} }_{HS} = \bnorm{ \widetilde{\Lambda} }_{HS}$
  and
  $\bnorm{ \widetilde{A} }_{op} = \bnorm{ \widetilde{\Lambda} }_{op}$.
  Now
  \begin{equation*}
    \begin{split}
      W_1\left(X,\frac{\bnorm{ \widetilde{A} }_{HS}}{\sqrt{n^2-1}}
        g\right) & = \sup_{\abs{f}_L\le 1} \abs{\E f(X) - \E f\left(\frac{\bnorm{
            \widetilde{A} }_{HS}}{\sqrt{n^2-1}}
          g\right)} \\
      & = \sup_{\abs{f}_L\le 1} \abs{\E \left(\E \left[ f(X) - f\left(\frac{\bnorm{
              \widetilde{A} }_{HS}}{\sqrt{n^2-1}}
            g\right) \middle\vert \Lambda \right]\right)} \\
      & \le \E \sup_{\abs{f}_L\le 1} \abs{\E \left[ f(X) - f\left(\frac{\bnorm{
              \widetilde{A} }_{HS}}{\sqrt{n^2-1}}
            g\right) \middle\vert \Lambda \right]} \\
      & \le \E \left(\frac{ 8}{\sqrt{n}} \frac{\bnorm{
          \widetilde{\Lambda}}_{op}^2}{\bnorm{
          \widetilde{\Lambda}}_{HS}}\right) \sum_{i=1}^d \frac{1}{\srank{B_i}},
    \end{split}
  \end{equation*}
  by Theorem \ref{T:DL-T0-C}. 

  Next note that
  \begin{equation*}
    \begin{split}
      W_1 \left(\frac{\bnorm{\widetilde{A}}_{HS}}{\sqrt{n^2-1}}
        g, \frac{\E \bnorm{\widetilde{A}}_{HS}}{\sqrt{n^2-1}}
        g\right)
      & = \sup_{\abs{f}_L\le 1} \abs{\E f \left(\frac{\bnorm{\widetilde{A}}_{HS}}{\sqrt{n^2-1}} g\right) - \E f \left(\frac{\E \bnorm{\widetilde{A}}_{HS}}{\sqrt{n^2-1}} g\right) } \\
      & \le \E \norm{\frac{\bnorm{\widetilde{A}}_{HS}}{\sqrt{n^2-1}} g - \frac{\E \bnorm{\widetilde{A}}_{HS}}{\sqrt{n^2-1}} g} \\
      & = \frac{\E \left(\abs{\bnorm{\widetilde{A}}_{HS} - \E \bnorm{\widetilde{A}}_{HS}} \norm{g}\right)}{\sqrt{n^2-1}} \\
      & = \frac{\E \abs{\bnorm{\widetilde{A}}_{HS} - \E \bnorm{\widetilde{A}}_{HS}} \E \norm{g}}{\sqrt{n^2-1}} \le \frac{\sqrt{d}\E \abs{\bnorm{ \widetilde{A}}_{HS} - \E \bnorm{\widetilde{A}}_{HS}}}{\sqrt{n^2-1}},
    \end{split}
  \end{equation*}
  where we have used the independence of $g$ and $A$ in the last
  equality.  The second statement of the theorem now follows from the
  triangle inequality for $W_1$ together with renormalization of $X$
  and $\frac{\E \lVert \widetilde{A}\rVert_{HS}}{\sqrt{n^2-1}}g$ by
  $\frac{\sqrt{n^2-1}}{\E\bnorm{\widetilde{A}}_{HS}}$.
\end{proof}

\begin{proof}[Proof of Corollary \ref{T:invariant-C}]
  The assumptions on the distribution of $A$ imply that the
  distribution of satisfies a logarithmic Sobolev inequality, and
  hence a strong concentration of measure property; cf.\ \cite[Section
  5.1]{Ledoux}.  In particular, for any $1$-Lipschitz function
  $F:M_n^{sa}(\C) \to \R$ (with respect to the Hilbert--Schmidt norm),
  \begin{equation}
    \label{E:lsi}
  \Prob\left[ F(A) - \E F(A) \ge t\right] \le e^{- n \alpha t^2/2}
  \end{equation}
  for all $t > 0$.  From this it can be proved that
  \[
  \beta_1 \sqrt{n} \le \E \bnorm{\widetilde{A}}_{HS} \le \beta_2
  \sqrt{n}
  \qquad \text{and} \qquad
  \gamma_1 \le \E \bnorm{\widetilde{A}}_{op} \le \gamma_2,
  \]
  where $\beta_1, \beta_2, \gamma_1, \gamma_2 > 0$ depend only on
  $\alpha$ (see \cite{MeSz}).  (For simplicity of exposition, in the
  remainder of of this proof, all constants may depend on $\alpha$.)

  It follows directly from \eqref{E:lsi} that
  \[
  \E \abs{\bnorm{\widetilde{A}}_{HS} - \E \bnorm{\widetilde{A}}_{HS}}
  \le C.
  \]
  It therefore suffices to show that
  \begin{equation}
    \label{E:norm-ratio}
  \E
  \left(\frac{\bnorm{\widetilde{A}}_{op}^2}{\bnorm{\widetilde{A}}_{HS}}\right)
    \le \frac{C}{\sqrt{n}}.
  \end{equation}

  Firstly, for $t \ge \frac{8\gamma_2^2}{\beta_1\sqrt{n}}$,
  \begin{equation}
    \label{E:middle-t}
    \begin{split}
      \Prob
      \left[\frac{\bnorm{\widetilde{A}}_{op}^2}{\bnorm{\widetilde{A}}_{HS}}
        \ge t \right] 
      & = \Prob \left[ \bnorm{\widetilde{A}}_{op}^2
        \ge t
        \bnorm{\widetilde{A}}_{HS} \right] \\
      & \le \Prob \left[ \bnorm{\widetilde{A}}_{op}^2 \ge \frac{\beta_1 \sqrt{n} }{2}
        t\right] +
      \Prob \left[ \bnorm{\widetilde{A}}_{HS} < \frac{\beta_1 \sqrt{n} }{2}
        \right]
      \\
      & \le \Prob \left[ \bnorm{\widetilde{A}}_{op} - \E
        \bnorm{\widetilde{A}}_{op} \ge \left(\frac{\beta_1 \sqrt{n}
          }{2} t\right)^{1/2} - \gamma_2\right] \\
      & \qquad + \Prob\left[ - \bnorm{\widetilde{A}}_{HS} + \E
        \bnorm{\widetilde{A}}_{HS} > \frac{\beta_1 \sqrt{n} }{2}\right] \\
      & \le e^{-\alpha \beta_1 n^{3/2}t/8} + e^{-\alpha \beta_1^2 n^2/8}.
    \end{split}
  \end{equation}
  Next, since $\bnorm{\widetilde{A}}_{op} \le
  \bnorm{\widetilde{A}}_{HS}$,
  \begin{equation}
    \label{E:big-t}
    \begin{split}
      \Prob
      \left[\frac{\bnorm{\widetilde{A}}_{op}^2}{\bnorm{\widetilde{A}}_{HS}}
        \ge t \right] 
      \le \Prob \left[ \bnorm{\widetilde{A}}_{op}
        \ge t \right]
      \le \Prob \left[ \bnorm{\widetilde{A}}_{op} - \E
        \bnorm{\widetilde{A}}_{op} \ge t - \gamma_2\right] 
      \le e^{-\alpha n t^2 / 8}
    \end{split}
  \end{equation}
  for $t \ge \frac{1}{2} \gamma_2$.

  We now estimate
  \[
  \E
  \left(\frac{\bnorm{\widetilde{A}}_{op}^2}{\bnorm{\widetilde{A}}_{HS}}\right)
  = \int_0^\infty \Prob
  \left[\frac{\bnorm{\widetilde{A}}_{op}^2}{\bnorm{\widetilde{A}}_{HS}}
    \ge t \right] \ dt
  \]
  using \eqref{E:middle-t} to bound the integrand for
  $\frac{8\gamma_2^2}{\beta_1\sqrt{n}} \le t \le \frac{1}{2}
  \gamma_2$,
  \eqref{E:big-t} for $t \ge \frac{1}{2} \gamma_2$, and the trivial
  upper bound of $1$ for
  $0 \le t \le \frac{8\gamma_2^2}{\beta_1\sqrt{n}}$.
\end{proof}

\section{Diagonal entries: proof of Theorem \ref{T:Schur-Horn}}
\label{S:proofs-Schur-Horn}

\begin{proof}[Proof of Theorem \ref{T:Schur-Horn}]
  Assume without loss of generality that for all $n$,
  $\tr \Lambda_n = 0$; this only amounts to writing $\Lambda_n$
  instead of $\widetilde{\Lambda_n}$.  In this case
  $\int x \ d\mu_n(x)=0$ for each $n$ (and so $m=0$).  Let
\[
\sigma_n^2 = \int x^2 \ d\mu_n(x) = \frac{1}{n^2} \norm{\Lambda_n}_{HS}^2
\qquad \text{and} \qquad
\sigma^2 = \int x^2 \ d\mu(x);
\]
because we have assumed that $\mu_n \to \mu$ in $W_2$, we have that
$\sigma_n \to \sigma$.

First consider the mean measure $\E \nu_n$.  Given any test
function $f:\R \to \R$,
\[
  \E \int f \ d\nu_n = \frac{1}{n} \sum_{i=1}^n \E f(a_{ii}^{(n)}).
\]
Now for any $i$, $a_{ii}^{(n)} = \inprod{A_n e_i}{e_i} =
\inprod{\Lambda_n U_n^* e_i}{U_n^* e_i}$, and $U_n^* e_i$ is
distributed uniformly on the unit sphere in $\C^n$.  Therefore,
\[
  \E \int f \ d\nu_n = \E f(\inprod{\Lambda_n Z}{Z}),
\]
where $Z$ is uniformly distributed on the unit sphere in $\C^n$; that
is, $\E \nu_n$ is precisely the distribution of
$\inprod{\Lambda_n Z}{Z}$.  It follows immediately from the $d=1$ case
of Theorem \ref{T:R1-T0-C} that
\begin{equation}
  \label{E:nu_n-gaussian}
  W_1\left(\inprod{\Lambda_n Z}{Z},\sigma_n \sqrt{\frac{n}{n+1}} g\right)\le
  C \frac{\norm{\Lambda_n}_4^2}{n^{2}\sigma_n},
\end{equation}
making use of the fact that 
$\norm{\Lambda_n}_{HS}^2=n^2\sigma_n^2$.

We now apply the concentration of measure phenomenon on
$\Unitary{n}$.  Note that 
if $A = U \Lambda U^*$, $B = V \Lambda V^*$ for $U, V \in
\Unitary{n}$,  then
\begin{equation}
  \label{E:diagonal-Lipschitz}
  \begin{split}
  \sqrt{\sum_{i=1}^n \abs{a_{ii} - b_{ii}}^2}
  & \le \norm{A - B}_{HS}
  = \norm{U \Lambda (U - V)^* + (V - U) \Lambda V^*}_{HS} 
  \\ & \le 2 \norm{\Lambda}_{op} \norm{U-V}_{HS}.
\end{split}
\end{equation}
If $f:\R \to \R$ is a $1$-Lipschitz test function, then it follows that
\[
  \abs{\frac{1}{n} \sum_{i=1}^n f(a_{ii}) - \frac{1}{n} \sum_{i=1}^n
  f(b_{ii})}
  \le \frac{1}{n} \sum_{i=1}^n \abs{a_{ii} - b_{ii}}
  \le \frac{2}{\sqrt{n}} \norm{\Lambda}_{op} \norm{U-V}_{HS};
\]
that is, if $\nu=\frac{1}{n}\sum_{j=1}^n\delta_{a_{jj}}$ for
$A=U\Lambda U^*$, then $U\mapsto \int f \ d\nu$ is a $\frac{2}{\sqrt{n}}
\norm{\Lambda}_{op}$-Lipschitz function of $U$.
Then \eqref{E:unitary-concentration} implies that
\begin{equation}
\label{E:dnu-concentration}
\Prob\left[\abs{\int f d\nu_n-\E \int f d\nu_n}\ge t\right]
\le
2\exp\left[-\frac{n^2t^2}{48\norm{\Lambda_n}^2_{op}}\right].
\end{equation}

Suppose now that $\norm{\Lambda_n}_{op} = o(n)$. Then
$\norm{\Lambda_n}_4^2 \le \norm{\Lambda_n}_{HS} \sqrt{\Lambda_n}_{op}
= n \sigma_n \norm{\Lambda_n}_{op}$,
and so by \eqref{E:nu_n-gaussian} and the fact that
$\sigma_n\to\sigma$,
\[
\int fd\nu_n\to\E f(\sigma g)
\]
for every Lipschitz function $f:\R \to \R$.  It follows from
\eqref{E:dnu-concentration} that for some $\eps(n) \to 0$,
\[
\Prob \left[ \abs{ \int f \ d\nu_n - \E f(\sigma{Z})} \ge  t + \eps(n)
\right]
\to 0
\]
for each fixed $t > 0$, so that $\nu_n \to N(0,\sigma)$ weakly in
probability.  Moreover, if $\norm{\Lambda_n}_{op}^2 = o
\bigl(\frac{n}{\sqrt{\log n}}\bigr)$, then
\[
\Prob \left[ \abs{ \int f \ d\nu_n - \E f(\sigma{Z})} \ge  \frac{7
    \norm{\Lambda_n}_{op}}{n} \sqrt{\log n} + \eps(n)
\right]
\le 2 n^{-49/48},
\] 
so $\nu_n \to N(0,\sigma)$ weakly almost surely by the Borel--Cantelli
Lemma.

\medskip

Next, by \eqref{E:diagonal-Lipschitz}, for each $i$,
$U \mapsto a_{ii}$ and $U \mapsto -a_{ii}$ are $2 \norm{\Lambda}_{op}$-Lipschitz functions on
$\Unitary{n}$, and so \eqref{E:unitary-concentration} implies that
\[
\Prob \left[|a_{ii}| \ge t \right] \le 2e^{-n t^2 / 48 \norm{\Lambda_n}_{op}^2}.
\]
Therefore
\[
\Prob \Bigl[\max_{1\le i \le n} \left|a_{ii}\right| \ge t \Bigr]
\le 2n  e^{-n t^2 / 48 \norm{\Lambda_n}_{op}^2},
\]
which implies that
\begin{equation*}
  \begin{split}
    \E \max_{1 \le i \le n} |a_{ii}| & \le \int_0^\infty \min\left\{1, 2n  e^{-n t^2 /
        48 \norm{\Lambda_n}_{op}^2}\right\} \ dt 
    \le C \norm{\Lambda_n}_{op} \sqrt{\frac{\log n}{n}},
  \end{split}
\end{equation*}
and by the Borel--Cantelli lemma, with probability 1,
\[
\max_{1\le i \le n} |a_{ii}| \le 10 \norm{\Lambda_n}_{op} \sqrt{\frac{\log
    n}{n}}
\]
for all sufficiently large $n$.

On the other hand, given $1 \le d \le n$, by Theorem
\ref{T:submatrix},
\[
W_1 \left(\frac{\sqrt{n^2-1}}{\norm{\Lambda_n}_{HS}}(a_{11}, \dots,
  a_{dd}), (g_1, \dots, g_d)\right) \le 9 d \sqrt{n}
\frac{\norm{\Lambda_n}_{op}^2}{\norm{\Lambda_n}_{HS}^2}
= \frac{9 d \norm{\Lambda_n}_{op}^2}{n^{3/2}\sigma_n^2}.
\]
Since $\max_{1\le i \le d} x_i$ is a $1$-Lipschitz function of $x \in
\R^d$, it follows that
\begin{equation*}
    \E \max_{1 \le i \le d} a_{ii}
    \ge \sigma_n \left(\E \max_{1\le
        i \le d} g_{i} -\frac{9 d \norm{\Lambda_n}_{op}^2}{n^{3/2}\sigma_n^2}
    \right).
\end{equation*}
It is well known that
$\E \max_{1 \le i \le d} g_i \ge c \sqrt{\log d}$. If we assume now
that $\norm{\Lambda_n}_{op} \le K \sqrt{n}$ then choosing
$d = \lfloor \sqrt{n} \rfloor$ completes the proof.
\end{proof}

\section*{Acknowledgements}

This research was partially supported by grants from the U.S.\
National Science Foundation (DMS-1612589 to E.M.)  and the Simons
Foundation (\#315593 to M.M.). Some of this work was carried out while
M.M.\ attended the Institut Henri Poincar\'e research trimester
``Analysis in Quantum Information Theory'', partially supported by NSF
grant DMS-1700168.  The authors thank Stanis{\l}aw Szarek and Michael
Walter for helpful discussions, Yan Fyodorov and Micha{\l} Oszmaniec
for references, and Ion Nechita particularly for outlining the
argument at the end of Section \ref{S:joint}.

\bibliographystyle{plain}
\bibliography{prescribed-eigenvalues}

\begin{thebibliography}{10}

\bibitem{AuSz}
G.~Aubrun and S.~Szarek.
\newblock {\em Alice and {B}ob Meet {B}anach. The Interface of Asymptotic
  Geometric Analysis and Quantum Information Theory}, volume 223 of {\em
  Mathematical Surveys and Monographs}.
\newblock American Mathematical Society, Providence, RI, 2017.

\bibitem{BeGe}
F.~Benaych-Georges.
\newblock Exponential bounds for the support convergence in the single ring
  theorem.
\newblock {\em J. Funct. Anal.}, 268(11):3492--3507, 2015.

\bibitem{BeZy}
I.~Bengtsson and K.~\.Zyczkowski.
\newblock {\em Geometry of Quantum States. An Introduction to Quantum
  Entanglement}.
\newblock Cambridge University Press, Cambridge, second edition, 2017.

\bibitem{Bhatia}
R.~Bhatia.
\newblock {\em Matrix Analysis}, volume 169 of {\em Graduate Texts in
  Mathematics}.
\newblock Springer-Verlag, New York, 1997.

\bibitem{ChMe}
S.~Chatterjee and E.~Meckes.
\newblock Multivariate normal approximation using exchangeable pairs.
\newblock {\em ALEA Lat. Am. J. Probab. Math. Stat.}, 4:257--283, 2008.

\bibitem{Collins}
B.~Collins.
\newblock Moments and cumulants of polynomial random variables on unitary
  groups, the {I}tzykson-{Z}uber integral, and free probability.
\newblock {\em Int. Math. Res. Not.}, (17):953--982, 2003.

\bibitem{CoNe-rqc}
B.~Collins and I.~Nechita.
\newblock Random quantum channels {I}: graphical calculus and the {B}ell state
  phenomenon.
\newblock {\em Comm. Math. Phys.}, 297(2):345--370, 2010.

\bibitem{CoSn}
B.~Collins and P.~\'{S}niady.
\newblock Integration with respect to the {H}aar measure on unitary, orthogonal
  and symplectic group.
\newblock {\em Comm. Math. Phys.}, 264(3):773--795, 2006.

\bibitem{Dal}
S.~Dallaporta.
\newblock Eigenvalue variance bounds for {W}igner and covariance random
  matrices.
\newblock {\em Random Matrices Theory Appl.}, 1(3):1250007, 28, 2012.

\bibitem{Deift}
P.~Deift.
\newblock {\em Orthogonal Polynomials and Random Matrices: a
  {R}iemann-{H}ilbert Approach}, volume~3 of {\em Courant Lecture Notes in
  Mathematics}.
\newblock New York University, Courant Institute of Mathematical Sciences, New
  York; American Mathematical Society, Providence, RI, 1999.

\bibitem{DeGi}
P.~Deift and D.~Gioev.
\newblock {\em Random Matrix Theory: Invariant Ensembles and Universality},
  volume~18 of {\em Courant Lecture Notes in Mathematics}.
\newblock Courant Institute of Mathematical Sciences, New York; American
  Mathematical Society, Providence, RI, 2009.

\bibitem{DoSt}
C.~D\"obler and M.~Stolz.
\newblock Stein's method and the multivariate {CLT} for traces of powers on the
  classical compact groups.
\newblock {\em Electron. J. Probab.}, 16:no. 86, 2375--2405, 2011.

\bibitem{Feze}
J.~Feinberg and A.~Zee.
\newblock Non-{G}aussian non-{H}ermitian random matrix theory: phase transition
  and addition formalism.
\newblock {\em Nuclear Phys. B}, 501(3):643--669, 1997.

\bibitem{Fol}
G.~Folland.
\newblock {\em Real Analysis. Modern Techniques and Their Applications}.
\newblock Pure and Applied Mathematics (New York). John Wiley \& Sons, Inc.,
  New York, second edition, 1999.

\bibitem{FySo}
Y.~Fyodorov and H.-J. Sommers.
\newblock Random matrices close to {H}ermitian or unitary: overview of methods
  and results.
\newblock {\em J. Phys. A}, 36(12):3303--3347, 2003.
\newblock Random matrix theory.

\bibitem{GuKrZe}
A.~Guionnet, M.~Krishnapur, and O.~Zeitouni.
\newblock The single ring theorem.
\newblock {\em Ann. of Math. (2)}, 174(2):1189--1217, 2011.

\bibitem{GuMa}
A.~Guionnet and M.~Ma\"{\i}da.
\newblock A {F}ourier view on the {$R$}-transform and related asymptotics of
  spherical integrals.
\newblock {\em J. Funct. Anal.}, 222(2):435--490, 2005.

\bibitem{GuZe-lda}
A.~Guionnet and O.~Zeitouni.
\newblock Large deviations asymptotics for spherical integrals.
\newblock {\em J. Funct. Anal.}, 188(2):461--515, 2002.

\bibitem{GuZe-srt}
A.~Guionnet and O.~Zeitouni.
\newblock Support convergence in the single ring theorem.
\newblock {\em Probab. Theory Related Fields}, 154(3-4):661--675, 2012.

\bibitem{Harish-Chandra}
Harish-Chandra.
\newblock Differential operators on a semisimple {L}ie algebra.
\newblock {\em Amer. J. Math.}, 79:87--120, 1957.

\bibitem{JoSm}
C.~H. Joyner and U.~Smilansky.
\newblock A random walk approach to linear statistics in random tournament
  ensembles.
\newblock {\em Electron. J. Probab.}, 23:Paper No. 80, 37, 2018.

\bibitem{JoSmWe}
C.~H. Joyner, U.~Smilansky, and H.~A. Weidenm\"{u}ller.
\newblock Spectral statistics of the uni-modular ensemble.
\newblock {\em J. Phys. A}, 50(38):385101, 35, 2017.

\bibitem{KeMeSi}
J.~P. Keating, F.~Mezzadri, and B.~Singphu.
\newblock Rate of convergence of linear functions on the unitary group.
\newblock {\em Journal of Physics A: Mathematical and Theoretical},
  44(3):035204, 2010.

\bibitem{Ledoux}
M.~Ledoux.
\newblock {\em The Concentration of Measure Phenomenon}, volume~89 of {\em
  Mathematical Surveys and Monographs}.
\newblock American Mathematical Society, Providence, RI, 2001.

\bibitem{MaOlAr}
A.~Marshall, I.~Olkin, and B.~Arnold.
\newblock {\em Inequalities: Theory of Majorization and its Applications}.
\newblock Springer Series in Statistics. Springer, New York, second edition,
  2011.

\bibitem{EM-linear}
E.~Meckes.
\newblock Linear functions on the classical matrix groups.
\newblock {\em Trans. Amer. Math. Soc.}, 360(10):5355--5366, 2008.

\bibitem{EM-Stein-multi}
E.~Meckes.
\newblock On {S}tein's method for multivariate normal approximation.
\newblock In {\em High Dimensional Probability {V}: the {L}uminy Volume},
  volume~5 of {\em Inst. Math. Stat. (IMS) Collect.}, pages 153--178. Inst.
  Math. Statist., Beachwood, OH, 2009.

\bibitem{EM-eigenfunctions}
E.~Meckes.
\newblock On the approximate normality of eigenfunctions of the {L}aplacian.
\newblock {\em Trans. Amer. Math. Soc.}, 361(10):5377--5399, 2009.

\bibitem{EM-book}
E.~Meckes.
\newblock {\em The Random Matrix Theory of the Classical Compact Groups},
  volume 218 of {\em Cambridge Tracts in Mathematics}.
\newblock Cambridge University Press, Cambridge, 2019.

\bibitem{MeMe-concentration}
E.~Meckes and M.~Meckes.
\newblock Concentration and convergence rates for spectral measures of random
  matrices.
\newblock {\em Probab. Theory Related Fields}, 156(1-2):145--164, 2013.

\bibitem{MeMe-powers}
E.~Meckes and M.~Meckes.
\newblock Spectral measures of powers of random matrices.
\newblock {\em Electron. Commun. Probab.}, 18:no. 78, 13, 2013.

\bibitem{MeSz}
M.~Meckes and S.~Szarek.
\newblock Concentration for noncommutative polynomials in random matrices.
\newblock {\em Proc. Amer. Math. Soc.}, 140(5):1803--1813, 2012.

\bibitem{Nechita}
I.~Nechita.
\newblock Asymptotics of random density matrices.
\newblock {\em Ann. Henri Poincar\'e}, 8(8):1521--1538, 2007.

\bibitem{NiSp}
A.~Nica and R.~Speicher.
\newblock {\em Lectures on the Combinatorics of Free Probability}, volume 335
  of {\em London Mathematical Society Lecture Note Series}.
\newblock Cambridge University Press, Cambridge, 2006.

\bibitem{OsAuGoKoAcLe}
M.~Oszmaniec, R.~Augusiak, C.~Gogolin, J.~Ko\l{}ody\ifmmode~\acute{n}\else
  \'{n}\fi{}ski, A.~Ac\'{\i}n, and M.~Lewenstein.
\newblock Random bosonic states for robust quantum metrology.
\newblock {\em Phys. Rev. X}, 6:041044, Dec 2016.

\bibitem{OsKu}
M.~Oszmaniec and M.~Ku\ifmmode~\acute{s}\else \'{s}\fi{}.
\newblock Fraction of isospectral states exhibiting quantum correlations.
\newblock {\em Phys. Rev. A}, 90:010302, Jul 2014.

\bibitem{PaSh}
L.~Pastur and M.~Shcherbina.
\newblock {\em Eigenvalue Distribution of Large Random Matrices}, volume 171 of
  {\em Mathematical Surveys and Monographs}.
\newblock American Mathematical Society, Providence, RI, 2011.

\bibitem{ReRo}
G.~Reinert and A.~R\"ollin.
\newblock Multivariate normal approximation with {S}tein's method of
  exchangeable pairs under a general linearity condition.
\newblock {\em Ann. Probab.}, 37(6):2150--2173, 2009.

\bibitem{RuVe}
M.~Rudelson and R.~Vershynin.
\newblock Invertibility of random matrices: unitary and orthogonal
  perturbations.
\newblock {\em J. Amer. Math. Soc.}, 27(2):293--338, 2014.

\bibitem{Saloff-Coste}
L.~Saloff-Coste.
\newblock Precise estimates on the rate at which certain diffusions tend to
  equilibrium.
\newblock {\em Math. Z.}, 217(4):641--677, 1994.

\bibitem{SaFySo}
D.~Savin, Y.~Fyodorov, and H.-J. Sommers.
\newblock Reducing nonideal to ideal coupling in random matrix description of
  chaotic scattering: Application to the time-delay problem.
\newblock {\em Phys. Rev. E}, 63:035202, Feb 2001.

\bibitem{Stein}
C.~Stein.
\newblock The accuracy of the normal approximation to the distribution of the
  traces of powers of random orthogonal matrices.
\newblock Technical Report 470, Stanford University Dept. of Statistics, 1995.

\bibitem{Stolz}
M.~Stolz.
\newblock Stein's method and central limit theorems for {H}aar distributed
  orthogonal matrices: Some recent developments.
\newblock In Gerold Alsmeyer and Matthias L{\"o}we, editors, {\em Random
  Matrices and Iterated Random Functions}, pages 73--88, Berlin, Heidelberg,
  2013. Springer Berlin Heidelberg.

\bibitem{Villani}
C.~Villani.
\newblock {\em Optimal Transport, Old and New}, volume 338 of {\em Grundlehren
  der Mathematischen Wissenschaften [Fundamental Principles of Mathematical
  Sciences]}.
\newblock Springer-Verlag, Berlin, 2009.

\bibitem{Voiculescu}
D.~Voiculescu.
\newblock Limit laws for random matrices and free products.
\newblock {\em Invent. Math.}, 104(1):201--220, 1991.

\bibitem{WeFy}
Y.~Wei and Y.~Fyodorov.
\newblock On the mean density of complex eigenvalues for an ensemble of random
  matrices with prescribed singular values.
\newblock {\em J. Phys. A}, 41(50):502001, 9, 2008.

\bibitem{Wootters}
W.~K. Wootters.
\newblock Random quantum states.
\newblock {\em Foundations of Physics}, 20(11):1365--1378, Nov 1990.

\bibitem{ZySo}
K.~{\.Z}yczkowski and H.~Sommers.
\newblock Induced measures in the space of mixed quantum states.
\newblock {\em J. Phys. A}, 34(35):7111--7125, 2001.

\end{thebibliography}

\end{document}